\documentclass[]{article}
\usepackage{ifpdf}

\newcommand{\Title}{A Dirichlet Form approach to MCMC Optimal Scaling}
\usepackage[british]{babel}

\usepackage{calc, xspace}
\usepackage{amsmath, amsthm, amsfonts, euscript}
\usepackage{amssymb}
\usepackage[geometry]{ifsym}
\usepackage{mathtools} 

\usepackage[longnamesfirst]{natbib}
\setcitestyle{round}

\usepackage[utf8]{inputenc}
\usepackage{eulervm}

\usepackage{graphicx}
\ifpdf
    \usepackage[usenames,dvipsnames]{xcolor}
\else
    \usepackage{color}
\fi
 
\definecolor{linkcolor}{named}{Maroon}
\definecolor{citecolor}{named}{PineGreen}
\definecolor{urlcolor}{named}{RoyalPurple}
\definecolor{okcolor}{named}{OliveGreen}
\definecolor{alertcolor}{named}{BrickRed}

\ifpdf
 \DeclareGraphicsExtensions{.png,.jpg}%
\else
 \DeclareGraphicsExtensions{.eps,.ps,.png,.jpg}%
\fi

\usepackage[totalwidth=480pt,totalheight=680pt]{geometry}
 \usepackage[
 \ifpdf
   pdftex,
   pdfstartview=FitH,
   unicode,
 \fi
 ]{hyperref}
 \hypersetup{
 pdfauthor={Giacomo Zanella and Wilfrid S. Kendall and Mylene Bedard},
 pdftitle={\Title},
 pdfsubject={Optimal scaling using Dirichlet forms},
 }
 \hypersetup{pdfstartview=FitH}
 \hypersetup{citecolor=citecolor}
 \hypersetup{linkcolor=linkcolor}
 \hypersetup{urlcolor=urlcolor}
 \hypersetup{colorlinks=true}

\RequirePackage{WSKmaths}
\graphicspath{{image/}}

\newcommand{\lxw}{L^2_{(X,W)}}
\newcommand{\I}{\mathcal{I}}
\renewcommand{\L}{\mathcal{L}}
\newcommand{\1}{\mathbb{I}} 

\newcommand{\xN}{\normalfont{\textbf{x}_{1:N}}}
\newcommand{\xn}{\normalfont{\textbf{x}_{1:n}}}
\newcommand{\xNn}{\normalfont{\textbf{x}_{(N+1):n}}}
\newcommand{\X}{\normalfont{\textbf{X}}}
\newcommand{\Xtail}{\normalfont{\textbf{X}_{(n+1):\infty}}}
\newcommand{\xtail}{\normalfont{\textbf{x}_{(n+1):\infty}}}
\newcommand{\XN}{\normalfont{\textbf{X}_{1:N}}}
\newcommand{\Xn}{\normalfont{\textbf{X}_{1:n}}}
\newcommand{\XNn}{\normalfont{\textbf{X}_{(N+1):n}}}
\newcommand{\wN}{\normalfont{\textbf{w}_{1:N}}}
\newcommand{\wn}{\normalfont{\textbf{w}_{1:n}}}
\newcommand{\wNn}{\normalfont{\textbf{w}_{(N+1):n}}}
\newcommand{\W}{\normalfont{\textbf{W}}}
\newcommand{\WN}{\normalfont{\textbf{W}_{1:N}}}
\newcommand{\Wn}{\normalfont{\textbf{W}_{1:n}}}
\newcommand{\WNn}{\normalfont{\textbf{W}_{(N+1):n}}}

\newcommand{\cn}{c_n}

\newcommand{\x}{\underline{x}}
\newcommand{\xA}{\underline{x}_A}
\newcommand{\xB}{\underline{x}_B}
\newcommand{\w}{\underline{w}}
\newcommand{\wA}{\underline{w}_A}
\newcommand{\wB}{\underline{w}_B}

\newcommand{\da}{\Delta_A}
\newcommand{\dda}{\widetilde{\Delta}_A}
\newcommand{\db}{\Delta_B}

\DeclareMathOperator*{\CartesianProduct}{\bigtimes}

\newcommand{\eg}{\emph{e.g.}\xspace}
\newcommand{\ie}{\emph{i.e.}\xspace}
\newcommand{\iid}{\emph{i.i.d.}\xspace}
\newcommand{\scale}{\tau} 
\newcommand{\support}{\operatorname{supp}} 
\newcommand{\AR}{\textbf{AR}} 

\newcommand{\Capacity}{\operatorname{Cap}}
\newcommand{\Hilbert}{\operatorname{\mathcal{H}}}
\newcommand{\Polish}{\operatorname{\mathbb{F}}}
\newcommand{\Sobolev}{\operatorname{S}}

\title{\Title}
\author{%
\href{http://didattica.unibocconi.it/docenti/cv.php?rif=198545}{Giacomo Zanella}%
\footnote{GZ supported in part by EPSRC Research Grant EP/D002060; Email: \href{mailto:giacomo.zanella@unibocconi.it}{\texttt giacomo.zanella@unibocconi.it}}%
,
\href{http://warwick.ac.uk/wsk}{Wilfrid S. Kendall}%
\footnote{WSK supported by EPSRC Research Grant EP/K013939; Email: \href{mailto:w.s.kendall@warwick.ac.uk}{\texttt w.s.kendall@warwick.ac.uk}}%
, and %
\href{http://www.dms.umontreal.ca/~bedard/}{Myl\`ene B\'edard}%
\footnote{MB supported in part by EPSRC Research Grant EP/D002060; Email: \href{mailto:mylene.bedard@umontreal.ca}{\texttt mylene.bedard@umontreal.ca}
\newline
This is a theoretical research paper and as such, no new data were created during this study.
}%
}

\begin{document} 
 \maketitle

\begin{abstract}
\noindent
This paper 
shows how the theory of
Dirichlet forms 
can be used
to deliver proofs of
optimal scaling results for Markov chain Monte Carlo algorithms
(specifically, Metropolis-Hastings random walk samplers)
under regularity conditions
which are substantially weaker 
than those
required by the original approach (based on the use of infinitesimal generators).
The Dirichlet form methods have
the added advantage of providing an explicit construction of the underlying infinite-dimensional context.
In particular, this enables us directly to establish weak convergence to the relevant infinite-dimensional distributions.
\end{abstract}

\noindent
\emph{\href{http://www.ams.org/msc/msc2010.html}{2010 Mathematics Subject Classification}:} 60F05; 60J22, 65C05.
\\
\emph{Key words and phrases:} 
\\
\textsc{Dirichlet form;
infinite-dimensional stochastic processes;
asymptotic analysis for MCMC;
Markov chain Monte Carlo (MCMC);
Metropolis-Hastings Random Walk (MHRW) Sampler;
Mosco convergence;
scaling limits;
optimal scaling;
weak convergence.}

\section{Introduction}\label{sec:introduction}
%

Markov Chain Monte Carlo (MCMC) algorithms form a general and widespread computational methodology 
addressing the problem of drawing samples from complex and intractable probability distributions \citep{RobertCasella-2001,BrooksGelmanJonesMeng-2011}.
Because of their simplicity and their scalability to high-dimensional settings, 
MCMC algorithms are now routinely used in many fields to obtain approximations of integrals that could not be tackled by common numerical methods.
One of the simplest and most popular MCMC schemes, 
the `Metropolis-Hastings Random Walk'
 (MHRW) Algorithm 
generates a Markov chain as follows.
Let $\Omega$ and $\pi$ denote the state space and the density of the distribution of interest.
Given a current state $x$, the chain samples a proposed value $y$ from some symmetric transition kernel $Q(x,\cdot)$ 
and moves to the proposal $y$ with probability $a(x,y)=1\wedge\frac{\pi(y)}{\pi(x)}$ (otherwise staying at $x$).
The resulting Markov chain is reversible with respect to $\pi$.
It can be used to obtain approximate samples and to perform Monte Carlo integration using ergodic averages.
Note that there are many variant algorithms, for example the Metropolis-Adjusted Langevin Algorithm (MALA: \citealp{RobertsRosenthal-1998}).

\subsection{MCMC Optimal Scaling}
Because of the popularity of MCMC algorithms, quantitative and mathematically rigorous understanding of their behaviour is of considerable interest.
The 
framework of \emph{Optimal Scaling}
\citep{RobertsGelmanGilks-1997} provides an effective and powerful approach.
The idea is to consider a sequence of target distributions $\pi^{(n)}$ defined on state spaces \(\Omega^{(1)}\), \(\Omega^{(2)}\), \ldots of increasing dimensionality (typically $\Omega^{(n)}=\Reals^n$), 
and to study the behaviour of the resulting sequence of MCMC algorithms as $n\to\infty$.
One obtains a sequence of Markov chains $\textbf{X}^{(1)}$, $\textbf{X}^{(2)},\dots$,
where each $\textbf{X}^{(n)}=\big\{\textbf{X}^{(n)}(t)\;:\; t=0,1,2, \ldots\big\}$ is obtained from the chosen MCMC algorithm with target $\pi^{(n)}$.
Appropriate sequences of algorithms lead to non-trivial limiting behaviour of $\textbf{X}^{(n)}$, 
namely that a time-rescaled version of $\textbf{X}^{(n)}$ converges to a tractable and informative limiting process $\textbf{X}^\infty$.

The resulting asymptotic analysis provides valuable insight in two practically relevant ways.
Firstly, inspection of the time-rescaled version of $\textbf{X}^{(n)}$ leads to rigorous proofs of useful results about the computational complexity of the sequence of MCMC algorithms,
viewed as depending on the dimensionality of the integration space $\Omega^{(n)}$. 
The now-classical example is that of \citet{RobertsGelmanGilks-1997} (see also \citealp{RobertsRosenthal-1998}). 
Their results show that, for simple targets on $\Omega^{(n)}=\Reals^n$, MHRW needs $O(n)$ steps to explore the state space entirely.
By way of contrast, the more sophisticated MALA will take $O(n^{1/3})$ steps to explore the state space entirely \citep{RobertsRosenthal-2016}.
Secondly, optimal scaling results facilitate optimization of MCMC performance by providing clear and mathematically-based guidance on how to tune the parameters defining the proposal distribution $Q^{(n)}$.
In fact optimizing such parameters for fixed dimensional chains $\textbf{X}^{(n)}$ is a difficult problem, typically not admitting analytic solution,
whereas the limiting object $\textbf{X}^{\infty}$ is often simple enough to allow a neat analytical optimization.
This yields
guidance (\eg~optimal values for average acceptance rates) 
which
is widely used by practitioners, 
especially \emph{via} self-tuning or \emph{Adaptive} MCMC methodologies \citep{AndrieuThoms-2008,Rosenthal-2011}.

Originally \citet{RobertsGelmanGilks-1997} dealt with MHRW and independent, identically distributed (\iid) targets, 
namely $\Omega^{(n)}=\Reals^n$ and $\pi^{n}(x^{(n)})=\prod_{i=1}^n \pi(x^{(n)}_i)$ where $\pi$ is a suitably smooth univariate density function.
The \iid~assumption is restrictive; however there are many extensions showing that the relevant results (order of complexity and optimal average acceptance rate) hold with significantly greater generality.
These extensions include:
independent targets with different scales \citep{Bedard-2006a},
Gibbs random fields \citep{BreyerRoberts-2000},
exchangeable normals \citep{NealRoberts-2006},
elliptical densities \citep{SherlockRoberts-2009},
densities with bounded support \citep{NealRobertsYuen-2012}
and
infinite-dimensional distributions with interaction terms \citep{MattinglyPillaiStuart-2012}.

The Optimal Scaling framework is one of the most successful and practically useful ways of performing asymptotic analysis of MCMC methods in high-dimensions.
Indeed, optimal scaling results are not limited to the analysis of MHRW and MALA, 
but have been used to analyze and compare a wide variety of MCMC schemes: 
Hamiltonian Monte Carlo \citep{BeskosPillaiRobertsSanzSernaStewart-2013}, Pseudo-Marginal MCMC \citep{SherlockThieryRobertsRosenthal-2015},
multiple-try MCMC \citep{BedardDoucMoulines-2012} and many others.

\subsection{Contribution of this paper}
The key mathematical result underpinning optimal scaling results, regardless of the classes of targets and algorithms considered,
concerns the convergence of time-rescalings of the sequence of resulting Markov chains $\textbf{X}^{(n)}$.
Such convergence is usually expressed in the form of weak convergence of the first coordinate ${X}^{(n)}_1$ of the vector process $\textbf{X}^{(n)}$,
with the weak limit being a one-dimensional limiting diffusion process ${X}^{\infty}_1$ (typically a Langevin diffusion).
The main interest of Optimal Scaling results lies exactly in the high-dimensionality of the target distribution.
So it is arguable that focusing on the first component only is somewhat restrictive and undesirable,
insofar as it deflects attention from the genuine multivariate problem of interest.
Rather than focusing on one-dimensional marginals, it would be more satisfying to study the full joint distribution of $\textbf{X}^{(n)}$. 
To do so one has to embed the process $\textbf{X}^{(n)}$, originally living in $\Omega^{(n)}=\Reals^n$, into the limiting space $\Omega^{\infty}=\Reals^\infty$ 
(for example by allowing moves of only the first $n$ coordinates, while viewing the remaining coordinates as being static and drawn from equilibrium).
One then needs to prove the convergence of the whole stochastic process $\textbf{X}^{(n)}$ to the infinite-dimensional limiting stochastic process $\textbf{X}^{\infty}$.

\citet{RobertsGelmanGilks-1997} observe that it is not hard to extend classic optimal scaling results to the study of convergence of a finite and fixed number of components 
(\ie~$\textbf{X}^{(n)}_{1:k}$ converging to $\textbf{X}^{\infty}_{1:k}$ for fixed $k$ and $n$ going to infinity),
but this 
confines attention to
the joint distribution of $\textbf{X}^{(n)}$ for fixed \(n\).
The
approach using \cite{EthierKurtz-1986} results, 
based on uniform convergence of generators, does not easily
apply to the study of processes living on infinite-dimensional state spaces (\eg~it can be necessary to assume that the state space is locally compact).
Moreover such techniques typically require rather substantial regularity conditions (in terms of target density derivatives and their moments).


In this paper we propose a different probabilistic approach to MCMC Optimal Scaling, relying on infinite-dimensional Dirichlet Form theory \citep{MaRoeckner-1992} to prove the crucial convergence result.
The abstract and powerful theory of Dirichlet forms,
and specifically the notion of \citet{Mosco-1994} convergence,
allows us to work directly and naturally on the infinite dimensional space $\Reals^\infty$ while requiring only modest regularity assumptions.
In the following we will focus on the classic MHRW framework of \citet{RobertsGelmanGilks-1997}, 
proving convergence for the whole infinite-dimensional stochastic process under mild regularity assumptions 
(finite Fisher information and local H\"older and controlled growth of first derivative of log-density).
In MCMC scenarios 
the smoothness and tail-behaviour of the target can impact massively on the performance of the algorithm \citep{NealRobertsYuen-2012,RobertsTweedie-1996b};
therefore it is important to establish general conditions under which the Optimal Scaling asymptotic analysis is still valid.
The following results are relevant to the Computational Statistics community interested in a theoretical understanding of MCMC methods,
and also to the Stochastic Processes community interested in convergence of stochastic processes and applications of Dirichlet Form theory.
To the best of our knowledge, this is the first application of Mosco convergence to the analysis of MCMC methods,
and we expect that the proof strategies developed in this paper will be useful to people seeking to prove convergence of infinite-dimensional stochastic processes arising in MCMC and other applications.

\subsection{Organization of the paper}
Section \ref{sec:overview_and_results} defines the class of MCMC algorithms being considered, and briefly reviews relevant theoretical notions, 
including the notion of Mosco convergence of forms \citep{Mosco-1994} and weak convergence through Dirichlet forms \citep{Sun-1998}.
It also presents the main results of the paper, namely Mosco and weak convergence of the relevant infinite-dimensional processes.
Section \ref{sec:proof_of_mosco} establishes Mosco convergence, while Section \ref{sec:weak} deals with weak convergence (under somewhat stronger regularity conditions):
the existence of the limiting process is established in Appendix \ref{appendix:capacity}.
Finally Section \ref{sec:discussion} discusses possibilities for future work and compares our work to some recent results 
involving Optimal Scaling for infinite-dimensional distributions \citep{MattinglyPillaiStuart-2012} 
and Optimal Scaling under weak regularity of the target \citep{DurmusLeCorffMoulinesRoberts-2016}.

\section{Overview and main results}\label{sec:overview_and_results}
%
This paper focuses on Metropolis-Hastings random walk samplers based on a simple target, 
namely the joint distribution of a large independent sample taken from a fixed distribution satisfying modest regularity conditions.
Suppose the fixed distribution is given by $\pi(\d x)=f(x)\d x$, a probability measure on $\Reals$.
Assume $f(x)=e^{\phi(x)}$ (so that \(f\) is everywhere positive), 
satisfying a finite Fisher information condition
\begin{equation}\label{eq:fisher}
\I=\int_{-\infty}^{\infty}|\phi'(x)|^2 \, f(x)\;\d x<\infty\,,
\end{equation}
and assume that the potential $\phi$ is continuous and everywhere differentiable, 
with derivative $\phi'=(\log f)'$ satisfying 
the following combination of a local H\"older condition and a growth condition:
for some $k>0$, $0<\gamma<1$ and $\alpha>1$,
\begin{equation}\label{eq:holder}
\qquad|\phi'(x+v)-\phi'(x)|
\quad<\quad
k\,\max\{|v|^\gamma,|v|^\alpha\}\,,\qquad x,v\in \Reals\,.
\end{equation}
This combined growth / local H\"older condition is much less restrictive than a global H\"older regularity with exponent $\gamma$.
We do not believe that condition \eqref{eq:holder} is necessary for our results to hold: 
however it combines the merit of reasonable generality with the advantage of simplicity of expression.
Note that condition \eqref{eq:holder} suffices for establishing optimal scaling in an \(L^2\) sense; 
however the Dirichlet form approach presently needs to use a stronger Lipschitz condition in order to establish weak convergence 
(for
more details 
see
Section \ref{sec:results}).

The following notational conventions are used.
Upper case letters denote random variables and corresponding lower case letters denote possible realizations, \eg~$X_1$ and $x_1$.
By $\L(X_1)$ we mean the distribution (or law) of the random variable $X_1$, for example $\L(W_1)=\mathcal{N}(0,1)$.
Subscripts denote vector components, \eg~$\XN=(X_1,\dots,X_N)$ or $\wNn=(w_{N+1},\dots,w_n)$.
Finally, 
we interpret the evaluation of probability density functions on vectors multiplicatively:
if $f$ is a one-dimensional probability density then its evaluation at a vector \(\XN\) is interpreted as the product of the density evaluated at each component.
Thus for example $f(\XN)=f(X_1)\cdots f(X_N)$, while $f(\wNn)=f(w_{N+1})\cdots f(w_n)$.

\subsection{Metropolis-Hastings Random Walk Sampler}\label{sec:MHRW}
For each $n=1,2,\ldots$, let $\big\{\textbf{X}^{(n)}(t): t=0,1,2, \ldots\big\}$
be a Metropolis-Hastings Random Walk (MHRW) sampler on $\Reals^n$,
with target measure $\pi^{\otimes n}(\d x_1,\dots,\d x_n)$ 
and with proposal measure defined by using independent and identically distributed Gaussian proposals on each component.
The component proposals are taken to be $\mathcal{N}(0,\frac{\scale^2}{n})$, for fixed $\scale>0$.
We seek to understand the limiting behaviour of a time-rescaled version of $\textbf{X}^{(n)}$ as $n\to\infty$.

For the sake of convenience we interpret $\big\{\textbf{X}^{(n)}(t):t=0,1,2,\ldots\big\}$ as an infinite-dimensional stochastic process on $\Reals^\infty$ updating only the first $n$ components,
with the remaining components drawn independently from the target distribution \(\pi\) and held fixed in time. 
The state space $\Reals^\infty$ is equipped with the product topology and corresponding Borel $\sigma$-algebra,
and we choose the infinite product measure $\pi^{\otimes\infty}$ as invariant measure.
It 
will be
useful 
to note that $\Reals^\infty$ is a Polish space (\ie~separable and completely metrizable topological space).
For example it can be equipped with the metric $d(\textbf{x},\textbf{y})=\sum_{j=1}^\infty 2^{-j}\frac{|x_j-y_j|}{1+|x_j-y_j|}$, which induces the product topology.
However $\Reals^\infty$ is not a Banach space, because its topology cannot be derived from any norm (for discussion of the broader context here see \citealp[Chapter IV]{Conway-1994}; 
details about $(\Reals^\infty,\pi^{\otimes \infty})$ are discussed in \citealp[Section 3]{Eldredge-2012}).

Our attention is focussed on the following explicit construction of the first step of the MHRW, hence defining $\big\{\textbf{X}^{(n)}(t):t=0,1\big\}$ 
(extension of this explicit construction to all of the time-homogeneous Markov process $\big\{\textbf{X}^{(n)}(t):t=0,1,2\ldots\big\}$ follows immediately from the Markov property of \(\textbf{X}^{(n)}\), 
but will not be the focus of attention in the sequel).
Let $\X=(X_1,X_2,\ldots)$ be a sequence of independent and identically distributed random variables on $\Reals$ with $\mathbb{P}_{X_1}(\d x)=\pi(\d x)$, 
let $\W=(W_1,W_2,W_3...)$ be a sequence of independent and identically distributed standard normal random variables on $\Reals$ with standard Gaussian density $g$,
and let $U$ be a Uniform$(0,1)$ random variable.
We require $\X$, $\W$ and $U$ to be independent of each other.
The first step of the $n^\text{th}$ MHRW $\big\{\textbf{X}^{(n)}(t):t=0,1\big\}$ is defined on $(\Reals^\infty,\pi^{\otimes \infty})$ by
\begin{equation*}
\textbf{X}^{(n)}(0)=(X_1,...,X_n,X_{n+1},X_{n+2},\dots)\;,\qquad
\textbf{X}^{(n)}(1)=(X_1+A_n\frac{\scale}{\sqrt{n}} W_1 ,...,X_n+A_n\frac{\scale}{\sqrt{n}} W_n ,X_{n+1},X_{n+2}\dots)\;,
\end{equation*}
where $A_n$ equals $1$ if $U<a(\Xn,\Wn)$ and $0$ otherwise, with
\begin{equation}
a(\Xn,\Wn)\quad=\quad
1\wedge\frac{f(\Xn+\frac{\scale}{\sqrt{n}} \Wn)}{f(\Xn)} 
\quad=\quad
1\wedge\frac{f(X_1+\frac{\scale}{\sqrt{n}} W_1 )\cdots f(X_n+\frac{\scale}{\sqrt{n}} W_n )}{f(X_1)\cdots f(X_n)} 
\end{equation}
being the Metropolis-Hastings acceptance function designed to induce reversibility.
Thus, as \(n\) increases, \(X^{(n)}\) proposes smaller jumps extending over a larger number of dimensions.
In due course we will re-scale time so that the smaller jumps are proposed more frequently in compensation for their reduced size.
The key result of \citet{RobertsGelmanGilks-1997} then runs as follows.
\begin{thm}[\citealp{RobertsGelmanGilks-1997}, Theorem 1.1]\label{thm:RGG}
 Suppose that the probability density \(f\) of \(\pi\) is positive and \(C^2\), that $f'/f$ is Lipschitz continuous and that
 \begin{align}\label{eq:RGG1}
  \int_{-\infty}^\infty
  \left(\frac{f'(x)}{f(x)}\right)^8
  f(x){\d}x
  \quad&=\quad M \quad<\quad \infty\,,\\
  \label{eq:RGG2}
   \int_{-\infty}^\infty
  \left(\frac{f''(x)}{f(x)}\right)^4
  f(x){\d}x
  \quad&<\quad \infty\,.
  \end{align}
Let \(U^n_t = X^{(n)}_1(\lfloor nt\rfloor)\), the first component of \(\textbf{X}^{(n)}\) at the re-scaled time \(\lfloor nt\rfloor\).
Then \(U^{(n)}\Rightarrow U\) as \(n\to\infty\), where \(U_0\) is distributed as \(\pi\), and \(U\) solves the stochastic differential equation
\begin{equation}\label{eq:RGG-SDE}
 {\d}U \quad=\quad \tau\sqrt{c(\tau)}{\d}B + \half \tau^2 c(\tau)\frac{f'(U)}{f(U)}{\d}t
\end{equation}
for \(c(\tau)=2F(-\tau\sqrt{\I}/2)\),
\(\I=\int_{-\infty}^\infty \left(f'/f\right)^2 f {\d}x\), where \(F\) is the standard normal distribution function.
\end{thm}
We shall show that the Dirichlet form approach allows us to replace the restrictive regularity and moment conditions of Theorem \ref{thm:RGG}
by \eqref{eq:fisher} and \eqref{eq:holder},
thus avoiding second-order conditions on \(f\) and concerns only weak growth and local H\"older conditions on \(\phi'=f'/f\),
as well as being an approach naturally adapted to the underlying infinite-dimensional framework. 

\subsection{Dirichlet forms}\label{sec:dirichlet}
%
Consider a Polish space \(\Polish\) furnished with a probability measure \(\mu\).
In the following we will be interested in \(\Polish=\Reals^\infty\) and \(\mu=\pi^{\otimes \infty}\) (for \(\pi\) as given at the beginning of Section \ref{sec:overview_and_results}).

We now recall some notions from the literature of Dirichlet forms
(for more details see \citealp{MaRoeckner-1992}).
Note that the general theory of Dirichlet forms applies even if \(\mu\) is merely a \(\sigma\)-additive measure, rather than a probability measure.
However we will describe results only in the case of a probability measure, which reduces the complexity required in the following definitions.

Let $\Hilbert$ be the Hilbert space $\Hilbert=L^2(\Polish,\mu)\).
For any $h$ and $v$ in $\Hilbert$, denote the usual $L^2$ inner product by $\langle h,v\rangle_{\Hilbert}=\int_{\Polish}h(x)v(x)\,\mu(\d x)$ 
and the related norm \(\|h\|_{\Hilbert}\) by $\|h\|_{\Hilbert}^2=\langle h,h\rangle_{\Hilbert}=\int_{\Polish}h(x)^2\,\mu(\d x)$.

A \emph{form} $\Phi$ on $\Hilbert$ is a non-negative definite and symmetric bilinear form $\Phi(h_1,h_2)$, 
defined for \(h_1\), \(h_2\) belonging to a dense linear subspace $D(\Phi)$ of $\Hilbert$, 
the \emph{domain} of $\Phi$ \citep[Section 1]{Mosco-1994}.
We will commit a mild abuse of notation by
using $\Phi(h)=\Phi(h,h)$ to denote
the associated quadratic functional,
and we will also refer to \(\Phi(h)\) as a form 
(the polarization identity yields a 1:1 correspondence between forms and quadratic functionals).
A form \(\Phi\) can be extended to the whole space $\Hilbert$ by setting $\Phi(h)=\infty$ for any $h\in \Hilbert\setminus D(\Phi)$.
A Dirichlet form is a closed, Markovian form \citep[Section 1]{Mosco-1994}:
its domain $D(\Phi)$ is complete under the inner product $\Phi(h_1,h_2)+\langle h_1,h_2\rangle_{\Hilbert}$
and moreover $\Phi(\tilde{h})\leq\Phi(h)$ when $\tilde{h}=(h \vee 0)\wedge 1 \in D(\Phi)$ for $h\in D(\Phi)$.

Given a Markov process on \(\Polish\), a Dirichlet form can be associated with it as follows.
In the discrete-time case,
let $\{\textbf{X}(t):t=0,1,\ldots\}$ be a discrete-time Markov chain on the Polish space $\Polish$, assumed reversible with respect to the probability measure $\mu$. 
The corresponding Dirichlet form $\left(\Phi,D(\Phi)\right)$ is given for \(h\in D(\Phi)=\Hilbert\) by
\begin{equation}\label{eq:dir_form_discrete}
\Phi(h)
\quad=\quad
\Expect{\Big(h\big(\textbf{X}(0)\big)-h\big(\textbf{X}(1)\big)\Big)h\big(\textbf{X}(0)\big)}
\,=\,
\frac{1}{2}\Expect{\Big(h\big(\textbf{X}(0)\big)-h\big(\textbf{X}(1)\big)\Big)^2}\,,
\end{equation}
with starting state $\textbf{X}(0)$ distributed according to $\mu$. 
Note that the second equality in \eqref{eq:dir_form_discrete} holds because of the reversibility assumption. 

Now consider the continuous-time case. 
Let $\{\textbf{X}^x(t):0\leq t < \infty\}$ be a continuous-time Markov process on $\Polish$, 
also reversible with respect to the measure $\mu$.
Here time is denoted by $t$, while $x$ is the starting point of the process.
Let $\{T_t:t\geq0\}$ denote the Markov semigroup of operators $T_t:\Hilbert\rightarrow \Hilbert$ 
given by $(T_th)(x)=\Expect{h\big(\textbf{X}^x(t)\big)}$ for $h\in \Hilbert$ and $x\in\Polish$.
The Dirichlet form $\left(\Phi,D(\Phi)\right)$ associated with $\{\textbf{X}^x(t):t\geq0\}$ (for \(x\in\Polish\)) is given by
\begin{equation}\label{eq:dir_form}
\Phi(h)
\quad=\quad
\lim_{t\downarrow 0}\frac{\langle (I-T_t)h,h\rangle_{\Hilbert}}{t}
\;=\;
\lim_{t\downarrow 0}\frac{1}{2}\frac{\langle (I-T_t)h,(I-T_t)h\rangle_{\Hilbert}}{t}\,,
\end{equation}
with $D(\Phi)$ being the subset of $\Hilbert$ for which the limit in \eqref{eq:dir_form} is finite.
Note that \eqref{eq:dir_form_discrete} can be obtained as a special case of \eqref{eq:dir_form},
by reformulating the discrete-time Markov chain as a continuous-time process with jumps happening according to an exponential clock of unit rate.

\citet{MaRoeckner-1992} show that, under some mild regularity conditions  
(for example regularity or quasi-regularity of the Dirichlet form in question; see Definition \ref{def:regularity} in Section \ref{sec:capacity} below), 
for each Dirichlet form $\Phi$ 
there exists a Markov process $\{\textbf{X}^x(t):t\geq0\}$ ($x\in\Polish$) such that $\Phi$ is its associated Dirichlet form.

\subsection{Mosco convergence of forms}\label{sec:mosco-convergence}
 \citet[Definition 2.1.1]{Mosco-1994} introduced the following notion of convergence of forms.
 In the case of Dirichlet forms, this entails uniform convergence of the semigroups of the associated processes:
see Theorem \ref{thm:convergence_of_semigroups} below.
\begin{defn}\label{defn:Mosco_convergence}
A sequence of forms $\{\Phi_n:n=1,2,\ldots\}$ in $\Hilbert$ converges to a form $\Phi$ in $\Hilbert$
(using the notation $\Phi_n\stackrel{M}\rightarrow\Phi$)
if the following conditions hold:
\begin{itemize}
\item[(M1)] For any $h, h_1, h_2, \ldots \in \Hilbert$ with $h_n\stackrel{w}\rightarrow h$ weakly in $\Hilbert$, it is the case that 
\[
\liminf_{n\to\infty} \Phi_n (h_n)\quad\geq\quad \Phi(h)\,; 
\]
\item[(M2)] For any $h\in \Hilbert$ there exists a sequence \(h_1\), \(h_2\), \ldots such that $h_n\rightarrow h$ (strongly) in $\Hilbert$ and 
\[
\limsup_{n\to\infty} \Phi_n (h_n)\quad\leq\quad \Phi(h)\,. 
\]
\end{itemize}
\end{defn}
\begin{rem}
 There is a potential terminological confusion between weak convergence of elements of a Hilbert space (\(h_n\stackrel{w}{\to}h\) if \(\langle h_n,g\rangle\to\langle h,g\rangle\) for all \(g\in\Hilbert\))
 and weak convergence of distributions of random variables (\(Z_n\Rightarrow Z\) if \(\Expect{f(Z_n)}\to\Expect{f(Z)}\) for all bounded continuous \(f\)).
 In the language of functional analysis, the second kind of convergence is more properly thought of as weak\({}^*\) convergence of (probability) measures.
 In this second case we will refer to (probabilistic) weak convergence.
\end{rem}

The following result plays a key enabling r\^ole in the application of Dirichlet forms to MCMC theory.
\begin{thm}\citep[Corollary 2.6.1]{Mosco-1994}\label{thm:convergence_of_semigroups}
Let $\Phi$ and $\Phi_n$ (for \(n=1,2,\dots\)) be Dirichlet forms on $\Hilbert$ with associated semigroups $\{\textbf{T}_t:t\geq0\}$ and $\{\textbf{T}^{(n)}_t: t\geq0\}$.
Then $\Phi_n\stackrel{M}\rightarrow\Phi$ if and only if the associated semigroups converge uniformly in the strong operator topology, meaning that
$\sup_{0<t\leq t_0} \big\|T^{(n)}_th-T_th\big\|_{\Hilbert}\rightarrow 0$ as \(n\to\infty\),  for any $t_0>0$ and $h\in \Hilbert$.
\end{thm}

\subsection{Nests, capacity and quasi-regularity}\label{sec:capacity}
We first introduce the notion of capacity (see \citealp[(2.2)]{AlbeverioRoeckner-1989} and \citealp[Def.III.2.1 and Ex.III.2.10]{MaRoeckner-1992})
\begin{defn}[Dirichlet form capacity]\label{def:capacity}
Given an open set \(U\subseteq \Polish\), we define the \emph{capacity} of $U$ as
\begin{equation}
\Capacity(U)
\;=\;
\inf\{\|h\|^2_{\Hilbert} + \Phi(h) \;:\; h\geq 1\text{ on }U\; \text{\(\mu\)-almost everywhere } \}\,,
\end{equation}
and, for general subsets $A\subseteq \Polish$,
\begin{equation}
\Capacity(A)
\;=\;
\inf\{\Capacity(U) \;:\; A\subseteq U\subseteq \Polish,\;U\text{ open}\}\,.
\end{equation}
\end{defn}

Let $\Phi$, $\Phi_1$, \(\Phi_2\) \ldots be Dirichlet forms on $\Hilbert=L^2(\Polish,\mu)$. 
The notion of $\Phi$-nests (\citealp[Def.III.2.1 and Thm.III.2.11]{MaRoeckner-1992})
is crucial when articulating the extent to which the Dirichlet forms are confined to suitable regions of \(\Polish\).

\begin{rem}
 In the following we denote by \(C_0(\Polish)\) the space of continuous functions of compact support on \(\Polish\), which is typically too small to be much use if \(\Polish\) is infinite-dimensional. 
\end{rem}
\begin{defn}[\(\Phi\)-nest of closed sets]
An increasing sequence of closed sets $F_1\subseteq F_2 \subseteq \ldots$ contained in \(\Polish\) is a \emph{$\Phi$-nest} if
$$
\lim_{k\rightarrow\infty}\Capacity_\Phi(\Polish\setminus F_k)\quad =\quad0\,.
$$
\end{defn}

\begin{defn}[Regular and quasi-regular Dirichlet forms, after \citealp{Schmuland-1994}]\label{def:regularity}
The Dirichlet form \(\Phi\) is \emph{regular} if \(D(\Phi)\cap C_0(\Polish)\) is dense in \(D(\Phi)\) with respect to the inner product \(\langle h_1,h_2\rangle_{\Hilbert}+\Phi(h_1,h_2)\)
and is dense in \(C_0(\Polish)\) with respect to the uniform norm. 
It is \emph{quasi-regular} if 
\begin{enumerate}
 \item there is a \(\Phi\)-nest of compact sets;
 \item there is a subset of \(D(\Phi)\), dense with respect to the inner product \(\langle h_1,h_2\rangle_{\Hilbert}+\Phi(h_1,h_2)\) and individually \(\Phi\)-quasi-continuous,
 in the sense that (an \(\mu\)-version
 of) any \(h\) in this subset is continuous in each closed set in a \(\Phi\)-nest (perhaps depending on \(h\));
 \item there is a countable subset of members of \(D(\Phi)\) with \(\Phi\)-quasi-continuous \(\mu\)-versions \(\tilde{u}_1\), \(\tilde{u}_2\), \ldots, such that
 \(\Polish\setminus N\) is separated by \(\tilde{u}_1\), \(\tilde{u}_2\), \ldots, for a set \(N\) which can be expressed as a subset of \(\bigcap_i F_i^c\) for some \(\Phi\)-nest \(F_1\subseteq F_2\subseteq\ldots\).
\end{enumerate}

\end{defn}

\begin{rem}
We assume that $1\in D(\Phi)$ and $1\in D(\Phi_n)$ for every $n$.
Such an assumption implies that the notions of quasi-regularity and nests are equivalent to their strict versions, namely strictly quasi-regular and strict nests \citep[Thm.V.2.15]{MaRoeckner-1992}.
This simplifies the exposition as it is then possible to ignore the strict versions of the above definitions.
\end{rem}

This brief summary concludes by introducing the notion of an increasing family of closed sets which is \emph{uniformly} a \(\Phi_n\)-nest for a sequence of Dirichlet forms \(\Phi_1\), \(\Phi_2\), \ldots.
\begin{defn}[Uniform \(\{\Phi_n\}\)-nest of closed sets]\label{def:uniform-nest}
An increasing sequence of closed sets $\{F_k\}_{k=1}^\infty$ contained in \(\Polish\) is a uniform $\{\Phi_n\}$-nest if
\begin{equation}\label{eq:unif_nest}
\lim_{k\rightarrow\infty}\sup_{n\in\Numbers}\Capacity_{\Phi_n}(\Polish\setminus F_k)\quad=\quad0\,.
\end{equation}
\end{defn}
Note that \cite{Sun-1998} refers to sequences satisfying \eqref{eq:unif_nest} as \emph{$\{\Phi_n\}$-nest} 
(or \emph{strict $\{\Phi_n\}$-nest}), while we prefer the more explicit expression \emph{uniform $\{\Phi_n\}$-nest}.

\subsection{Results of the paper}\label{sec:results}
This paper applies the above notions of Dirichlet forms in the context of the MHRW framework described in Section \ref{sec:MHRW},
based on \(\Polish=\Reals^\infty\) and \(\mu=\pi^{\otimes \infty}\).
For each \(n=1,2,\ldots\), consider the MHRW $\{\textbf{X}^{(n)}(t):t=1,2,\ldots\big\}$ subject to a time-rescaling by a factor of $n$.
Via \eqref{eq:dir_form_discrete}, this motivates consideration of the following Dirichlet form: 
\begin{equation}\label{eq:def_of_phi_n}
\Phi_n (h)\quad=\quad
\frac{n}{2}\;\Expect{\left(h(\textbf{X}^{(n)}(1))-h(\textbf{X}^{(n)}(0))\right)^2}
\qquad \text{ for }h\in \Hilbert\,.
\end{equation}
(This is the Dirichlet form corresponding to the continuous-time Markov process resulting from the MHRW 
reformulated as a discrete-time Markov chain jumping at instants of an exponential clock of rate \(n\).) 
The natural candidate for a limiting Dirichlet form (as \(n\to\infty\)) is given by
\begin{equation}\label{eq:def_of_phi}
\Phi (h)=\left\{
\begin{array}{ll}
\frac{1}{2}\scale^2c(\scale)\;\mathbb{E}[|\nabla h(\X)|^2] & \text{ for } h\in \Sobolev\,,\\
\infty                                                     & \text{ for } h\notin \Sobolev\,,
\end{array}
\right.
\end{equation}
where .
Here the domain \(\Sobolev\) of \(\Phi\) is precisely the region where the first expression in \eqref{eq:def_of_phi} can be viewed as finite. 
Accordingly, set $S=W^{1,2}(\Reals^\infty,\pi^{\otimes \infty})$ to be the Sobolev space defined as the closure of $\bigcup_{N> 0}C_{0,N}^\infty(\Reals^\infty)\subset \Hilbert$ according to the norm
\begin{multline}\label{eq:sobolev_norm}
\|h\|_{\Sobolev}^2
\quad=\quad
\int_{\Reals^N}
\left(
|h(x_{1:N})|^2
+
\sum_{i=1}^N\left|\frac{\partial}{\partial x_i}h(x_{1:N})\right|^2
\right)
\pi^{\otimes N}(\d x_{1:N})\,,
\\
\quad\text{ when }
h\in C_{0,N}^\infty(\Reals^\infty),
\text{ so }h(x)={h}(x_{1:N})\,.
\end{multline}
Here $C_{0,N}^\infty(\Reals^\infty)$ is the set of infinitely differentiable functions with compact support depending only on the first $N$ components.

The gradient $\nabla h$ in \eqref{eq:def_of_phi} is then defined 
as the continuous extension to \(\Sobolev\) of the natural definition of \(\nabla\) on \(\bigcup_{N> 0}C_{0,N}^\infty(\Reals^\infty)\).
So if \(h\in \Sobolev\) then the gradient \(\nabla h\)
is a measurable function from $\Reals^\infty$ to the Hilbert sequence space $\ell^2=\{x\in\Reals^\infty\,:\,\sum_{i=1}^{\infty}x_i^2<\infty\}$, 
satisfying the following properties:
\begin{enumerate}
\item $\Expect{\langle\nabla h(\X),\nabla h(\X)\rangle_{\ell^2}}<\infty$ 
and 
\item for any $i=1,2,\dots$, it is the case that
$\langle\nabla h(x),e^{(i)}\rangle_{\ell^2}=\frac{\partial}{\partial x_i}h(x)$ for $\pi^{\otimes\infty}$-almost every $x$, where $e^{(i)}\in\Reals^\infty$ with  $e^{(i)}_j=\delta_{i,j}$ the Kronecker delta.
\end{enumerate}
\citet[Equation (1.12) and Remark 1.12]{AlbeverioRoeckner-1989} show that such a function exists and is $\pi^{\otimes\infty}$-almost everywhere unique.
 
The Dirichlet form in \eqref{eq:def_of_phi} corresponds to an infinite-dimensional continuous-time Markov process $\{\textbf{X}^{\infty}(t):t\geq0\}$
 with state-space $(\Reals^\infty,\pi^{\otimes \infty})$,
for which each component evolves according to an independent copy of a specific diffusion on $\Reals$ with invariant measure $\pi$ and speed given by a specified function of $\scale$.
Some care is needed to establish a rigorous proof that such a process has associated Dirichlet form given in the form of \eqref{eq:def_of_phi}.
\citet[Equations (2.8)-(2.11)]{AlbeverioRoeckner-1989} give sufficient conditions on $\Phi$ for the corresponding Markov process to be well defined.
In Appendix \ref{appendix:capacity} we prove that these conditions hold for $\Phi$ as specified in \eqref{eq:def_of_phi}.
A simple computation with Gaussian densities shows that
\[
 c(\scale)\quad=\quad\Expect{1\wedge\exp\left(\mathcal{N}(-\tfrac{\scale^2}{2}\I,\scale^2\, \I)\right)}
 \quad=\quad
 2 F\left(-\half\scale \sqrt{\I}\right)\,,
\]
where \(F\) is the standard normal distribution function:
the limiting Dirichlet form \eqref{eq:def_of_phi} therefore agrees with the Dirichlet form for the limiting diffusion given by \citet{RobertsGelmanGilks-1997} as described in Theorem \ref{thm:RGG}.

The key result of this paper is that Mosco convergence of \(\Phi_n\) to \(\Phi\) holds under the relatively weak conditions on the potential \(\phi\) given at and above \eqref{eq:holder}
(finite Fisher information, and combined local H\"older and growth condition for the derivative of the potential \(\phi\)).
\begin{thm}\label{thm:conv_of_dir_forms}
For $\Phi_n$ and $\Phi$ defined by \eqref{eq:def_of_phi_n} and \eqref{eq:def_of_phi}, using a potential \(\phi\) satisfying \eqref{eq:holder}
together with finite Fisher information $\I=\int_{-\infty}^{\infty}|\phi'(x)|^2 \, f(x)\d x<\infty$,
it is the case that Mosco convergence $\Phi_n\stackrel{M}\rightarrow\Phi$ holds.
\end{thm}
\begin{proof}
It suffices to establish both (M1) and (M2) of Definition \ref{defn:Mosco_convergence} above.
Dealing with these in reverse order (so as to dispose of the easiest case first),
Property (M1) is established in Section \ref{sec:mosco_1} below, and
Property (M2) is established in Section \ref{sec:mosco_2}.
\end{proof}

Mosco convergence of forms immediately implies the uniform convergence of the associated semigroups such as $\{T_t:t\geq0\}$,
and hence (probabilistic) vague convergence of the finite-dimensional distributions of the corresponding process $\{\textbf{X}^{(n)}_t: t\geq0\}$ (finite dimensional,
in the sense of joint distribution of evaluations of the processes at a finite collection of time points).
\begin{cor}
Under the assumptions of Theorem \ref{thm:conv_of_dir_forms}, let $\{\textbf{X}^{\infty}(t):t\geq0\}$ and $\{\textbf{X}^{(n)}(t): t\geq0\}$ be the Markov processes associated with $\Phi$ and $\Phi_n$ and let $\{T_t: t\geq0\}$ and $\{T^{(n)}_t:t\geq0\}$ be their associated semigroups.
Then $\Phi_n\stackrel{M}\rightarrow\Phi$ implies the uniform convergence of semigroups in the strong  operator topology: for any $t_0>0$ and $h\in \Hilbert$
\begin{equation*}
\qquad\qquad\sup_{0<t\leq t_0} \big\|T^{(n)}_th-T_th\big\|_{\Hilbert}\longrightarrow 0
\qquad\hbox{as }n\to\infty\,.
\end{equation*}
\end{cor}
\begin{rem}
 \citet{Kolesnikov-2006} notes that vague convergence holds for finite-dimensional distributions of the corresponding Markov processes. 
 Note however that the above Corollary establishes \(L^2\) convergence of marginal distributions, which in some respects is much stronger 
 (\eg~it 
 controls 
 some
 unbounded test functions).
\end{rem}
\begin{proof}
Follows from Theorems \ref{thm:convergence_of_semigroups} and \ref{thm:conv_of_dir_forms}.
\end{proof}

These results lead to optimal scaling arguments for finite-dimensional distributions of the Metropolis-Hastings random walk sampler,
directly following the final argument of \citet{RobertsGelmanGilks-1997}. 
Fastest asymptotic exploration of the state space
is obtained exactly by optimizing the limiting process (governed by the Dirichlet form given in \eqref{eq:def_of_phi}).
This limiting Dirichlet form depends on \(\scale\) only through a multiplicative factor \(\scale^2c(\scale)\) which measures the speed at which the limiting process evolves;
therefore exploration occurs as fast as possible exactly when \(\scale^2c(\scale)=\Expect{1\wedge\exp\left(\mathcal{N}(-\tfrac{\scale^2}{2}\I,\scale^2\, \I)\right)}\) is maximized,
and at this maximum the acceptance probability for jumps is given by the famous ``Goldilocks constant'' \(0.234\) obtained by \citet{RobertsGelmanGilks-1997}.
See \citet{RobertsRosenthal-2016} for more details on the connection between asymptotic analysis through scaling limits and the algorithmic complexity of MCMC algorithms.

 Good practice in Markov-chain Monte Carlo involves estimators which make  use of entire sample paths (deleting the initial ``burn-in'' periods),
and so it is relevant to consider (probabilistic) weak convergence of the distribution of the entire sample path of \(\{\textbf{X}^{(n)}(t): t\geq0\}\) to that of \(\{\textbf{X}^\infty(t):t\geq0\}\).
\citet{Sun-1998} provides sufficient conditions to prove this 
using Dirichlet form theory.


\begin{thm}\label{thm:wei}\citep[Theorem 1]{Sun-1998}
Let $\Phi$ and $\Phi_n$ (for \(n=1,2,\dots\)) be quasi-regular Dirichlet forms on $\Hilbert=L^2(\Polish,\mu)$,
and let 
$\{\textbf{X}^\infty (t):t\geq0\}$ and
$\{\textbf{X}^{(n)}(t): t\geq0\}$ 
be their associated Markov processes, with random starting points $\textbf{X}^\infty(0)$ and $\textbf{X}^{(n)}(0)$ all distributed as \(\mu\). 
Suppose that
\begin{enumerate}
\item[(S1)] $\Phi_n\stackrel{M}\rightarrow\Phi$ (Definition \ref{defn:Mosco_convergence}), and
moreover a stronger form of condition (M2) of Definition \ref{defn:Mosco_convergence} applies; \(\limsup_{n\to\infty}\Phi_n(u)\leq\Phi(u)\)
(\ie~the sequence of \(u_n\) in (M2) may all be chosen equal to \(u\)).
\item[(S2)] Any $\Phi$-nest of compact sets is also a uniform $\{\Phi_n\}$-nest (Definition \ref{def:uniform-nest}).
\end{enumerate}
Then $\textbf{X}^{(n)}$ converges to $\textbf{X}^\infty$ in the sense of (probabilistic) weak convergence.
\end{thm} 
Note that the topology of \(\Polish\) only plays a role in formulating closedness and compactness of the sets $F_1$, $F_2$, \ldots.
The previous result, together with results from Section \ref{sec:proof_of_mosco}, can then be used to prove weak convergence of the process of interest,
so long as we strengthen the regularity required of the density \(f\) (and thence of the potential \(\phi\)).
\begin{thm}\label{thm:weak_conv_of_dir_forms}
Let $\{\textbf{X}^{(n)}(t): t\geq0\}$  and $\{\textbf{X}^\infty(t):t\geq0\}$ be the Markov processes associated with $\Phi_n$ and $\Phi$ defined by \eqref{eq:def_of_phi_n} and \eqref{eq:def_of_phi} (see Sections \ref{sec:MHRW} and \ref{sec:results}). Suppose that the potential \(\phi\) has Lipschitz-continuous first derivative, meaning that $|\phi'(x+v)-\phi'(x)|
<k|v|$ for a fixed $k$ and for all $x,v\in \Reals$, and finite Fisher information, meaning that $\I=\int_{-\infty}^{\infty}|\phi'(x)|^2 \, f(x)\d x<\infty$.
Then $\textbf{X}^{(n)}$ converges to $\textbf{X}^\infty$ in the sense of (probabilistic) weak convergence.
\end{thm}
\begin{rem}
Lipschitz continuity of $\phi'$ is required in order to allow use of Lemma \ref{lemma:capacity_bound} from Section \ref{sec:weak} below.
\end{rem}
\begin{proof}
The result follows by proving conditions (S1) and (S2) of Theorem \ref{thm:wei}.
Both conditions can be deduced from Theorem \ref{thm:conv_of_dir_forms} and Lemma \ref{lemma:capacity_bound} from Section \ref{sec:proof_of_mosco}, as follows.

First consider (S1).
Theorem \ref{thm:conv_of_dir_forms} guarantees $\Phi_n\stackrel{M}\rightarrow\Phi$ and
therefore it suffices to prove 
\[
\limsup_{n\to\infty}\Phi_n(u)\quad\leq\quad \Phi(u)\qquad \text{ for every }u\in\Hilbert\,.
\] 
This holds trivially if $\Phi(u)=\infty$, so suppose $\Phi(u)<\infty$.
Since $\Phi_n\stackrel{M}\rightarrow\Phi$, there exists a sequence $\{u_n\}\subset\Hilbert$ such that $u_n\rightarrow u$ in $\Hilbert$ and  \(\limsup_{n\to\infty}\Phi_n(u_n)\leq\Phi(u)\).
Moreover, using the construction described in Section \ref{sec:mosco_2}, such a sequence can be chosen such that $\Phi(u-u_n)\rightarrow 0$.
Then Lemma \ref{lemma:capacity_bound} of Section \ref{sec:weak} implies that
\(\Phi_n(u-u_n)\;\rightarrow\;0\),
because
\begin{equation}\label{eq:phi_n_to_zero}
\Phi_n(u-u_n)
\quad\leq\quad
c(\|u-u_n\|_{\Hilbert}+\Phi(u-u_n))
\quad\rightarrow\quad
0\,.
\end{equation}
Bilinearity of 
$$
\Phi_n(u,v)\quad=\quad
\frac{n}{2}\Expect{\left(u\left(\X^{(n)}(0)\right)-u\left(\X^{(n)}(1)\right)\right)\left(v\left(\X^{(n)}(0)\right)-v\left(\X^{(n)}(1)\right)\right)}
$$
for any $u,v\in\Hilbert$ permits the deduction that
$$
\Phi_n(u)\quad=\quad
\Phi_n(u_n+(u-u_n),u_n+(u-u_n))=\Phi_n(u_n,u_n)+\Phi_n(u-u_n,u-u_n)+2\Phi_n(u_n,u-u_n)\,.
$$
Therefore
\begin{equation}\label{eq:phi_n_triangular_ineq}
\limsup \Phi_n(u)
\;\leq\;
\limsup \Phi_n(u_n)
+
\limsup \Phi_n(u-u_n)
+
2\limsup \Phi_n(u_n,u-u_n)\,.
\end{equation}
As a consequence of \eqref{eq:phi_n_to_zero}, it follows that $\limsup \Phi_n(u-u_n)=0$.
Moreover an application of the Cauchy-Schwartz inequality and the fact that $\limsup \Phi_n(u_n)\leq\Phi(u)$ shows
$$
\limsup \Phi_n(u_n,u-u_n)
\;\leq\;
\limsup \sqrt{\Phi_n(u_n)\Phi_n(u-u_n)}
\;\leq\;
\sqrt{\Phi(u)}\limsup \sqrt{\Phi_n(u-u_n)}
\;=\;
0\,.
$$
When combined with \eqref{eq:phi_n_triangular_ineq} and (M2) of Mosco convergence,
the latter results in the deduction that $\limsup \Phi_n(u)\leq\Phi(u)$, as desired.

Now consider condition (S2).
Suppose $F_1 \subseteq F_2 \subseteq \ldots$ is a $\Phi$-nest of compact sets.
Therefore there exist $u_k\in D(\Phi)$ with $u_k\geq 1$ on $\Polish \setminus F_k$ such that
$\|u_k\|_{\Hilbert}+\Phi(u_k)\rightarrow 0$.
By definition of $\Capacity_{\Phi_n}$ and 
by Lemma \ref{lemma:capacity_bound} below, it is the case that
$$
\sup_{n} \Capacity_{\Phi_n}(\Polish\setminus F_k)
\quad\leq\quad
\sup_{n}\|u_k\|_{\Hilbert}+\Phi_n(u_k)
\quad\leq\quad
\|u_k\|_{\Hilbert}+C(\|u_k\|_{\Hilbert}^2+\Phi(u_k))
\quad\rightarrow\quad 0\,.
$$
Therefore $\{F_k\}_{k\in\Numbers}$ is a uniform $\{\Phi_n\}$-nest and so (S2) holds.
\end{proof}

\section{Mosco convergence for Metropolis-Hastings Random Walks}\label{sec:proof_of_mosco}
In this section we establish Mosco convergence in three steps.
We begin with a lemma and a corollary which describe central limit behaviour 
for a conditioned instance of the Metropolis-Hastings ratio, making heavy use of the regularity conditions at and above \eqref{eq:holder}.
This is then applied to establish the two conditions for Mosco convergence (Definition \ref{defn:Mosco_convergence}) in Sections \ref{sec:mosco_2} and \ref{sec:mosco_1}

\subsection{Convergence of the acceptance function}
Consider the Metropolis-Hastings ratio for the Metropolis-Hasting random walk algorithm, conditioned on the chain state.
Under mild conditions (finite Fisher information, local H\"older and controlled growth of derivative of log-density), 
we now show that the conditioned ratio $a(\Xn,\Wn)|\Xn=\xn$ 
converges in distribution to $1\wedge\exp\left(\mathcal{N}(-\tfrac{\scale^2}{2}\I,\scale^2 \, \I)\right)$ as \(n\to\infty\), for almost every sequence $(x_1,x_2,\dots)$.

\begin{lem}\label{lemma:acceptance_prob}
Let $\phi:\Reals\rightarrow\Reals$ and $\W=(W_1,W_2,\dots)$ be as described above in Section \ref{sec:overview_and_results}.
Given finite Fisher information, and local H\"older and controlled growth for the derivative of the log-density \(\phi\),
for $\pi^{\otimes\infty}$-almost every sequence $(x_1,x_2,\dots)$,
\begin{equation}\label{eq:acceptance_prob}
\log\left(\frac{f(\xn+\frac{\scale}{\sqrt{n}} \Wn)}{f(\xn)} \right)
\quad=\quad
\sum_{i=1}^n
\left( \phi\left(x_i+\frac{\scale}{\sqrt{n}} W_i\right)-\phi(x_i)  \right)
\quad\stackrel{\mathcal{D}}\rightarrow \quad
\mathcal{N}\left(-\frac{\scale^2}{2}\I,\scale^2 \,\I\right)\,.
\end{equation}
\end{lem}

\begin{proof}
Throughout the proof we condition implicitly on $X_1=x_1,X_2=x_2,\dots$.
We begin by separating the left-hand side of \eqref{eq:acceptance_prob} into two summands,
the first of which is of mean zero and carries all the asymptotic random variation.
\begin{equation}\label{eq:divide_acc_ratio}
\sum_{i=1}^n\left( \phi\left(x_i+\frac{\scale}{\sqrt{n}} W_i  \right)-\phi(x_i)  \right)
\;=\;
\frac{\scale}{\sqrt{n}}\sum_{i=1}^n\phi'(x_i)W_i+
\frac{\scale}{\sqrt{n}}\sum_{i=1}^nW_i\int_0^1
\left(\phi'\left(x_i+\frac{\scale u}{\sqrt{n}} W_i \right)-\phi'(x_i)\right) \d u\,.
\end{equation}
Analysis of the first summand of the right-hand side of \eqref{eq:divide_acc_ratio} can be achieved rapidly
using the strong law of large numbers:
$\tfrac{1}{n}{\sum_{i=1}^n\phi'(x_i)^2}$ converges to $\I$ for $\pi^{\otimes\infty}$-almost every $x_1,x_2,\dots$.
Since $\sigma_n\rightarrow\sigma$ implies $\mathcal{N}\left(0,\sigma_n\right)
\stackrel{\mathcal{D}}\rightarrow
N\left(0,\sigma\right)$, 
it follows that for $\pi^{\otimes\infty}$-almost every $x_1,x_2,\dots$
\begin{equation*}
\L\left(
\frac{\scale}{\sqrt{n}}\sum_{i=1}^n\phi'(x_i)W_i
\right)
\;=\;
\mathcal{N}\left(0,\scale^2\times \tfrac{1}{n}{\sum_{i=1}^n\phi'(x_i)^2}\right)
\quad\stackrel{\mathcal{D}}\rightarrow \quad
\mathcal{N}\left(0,\scale^2 \, \I\right)\,.
\end{equation*}

The second summand of the right-hand side of \eqref{eq:divide_acc_ratio} requires more detailed attention, 
and its treatment requires some regularity of \(\phi'\), for example as  expressed in \eqref{eq:holder} above.
We seek to show that this summand converges in distribution to $-\frac{\scale^2}{2}\I$.
The strategy is to show that its expectation converges to $-\frac{\scale^2}{2}\I$, while its variance vanishes asymptotically.
Recall that variances are bounded by second moments. 
Applying this to each of the \(n\) conditionally independent terms involved in the finite sum
(conditioning implicitly on  $X_1=x_1,X_2=x_2,\dots$ as noted above), we find:
\begin{multline}\label{eq:variance_sec_moment}
\Var{\frac{\scale}{\sqrt{n}}\sum_{i=1}^nW_i
\int_0^1\left(\phi'\left(x_i+\frac{\scale u}{\sqrt{n}} W_i \right)-\phi'(x_i) \right)\d u}
\\
\quad=\quad
\frac{\scale^2}{n}\sum_{i=1}^n \Var{W_i
\int_0^1\left(\phi'\left(x_i+\frac{\scale u}{\sqrt{n}} W_i \right)-\phi'(x_i) \right)\d u}
\;\\
\quad\leq\quad
\frac{\scale^2}{n}\sum_{i=1}^n
\Expect{\left|W_i
\left(\int_0^1\left(\phi'\left(x_i+\frac{\scale u}{\sqrt{n}} W_i \right)-\phi'(x_i) \right)\d u\right)\right|^2}
\;.
\end{multline}
Employing the regularity of $\phi'$ as given in the combined growth / local H\"older condition \eqref{eq:holder}, 
and noting that $u^\alpha\leq u^\gamma$ for $u\in (0,1)$ and $n^\gamma\leq n^\alpha$ for $n\geq 1$ (with $\alpha$ and $\gamma$ as given in \eqref{eq:holder}),
\begin{multline}\label{eq:holder_lemma}
\left|W_i\int_0^1
\left(\phi'\left(x_i+\frac{\scale u}{\sqrt{n}} W_i \right)-\phi'(x_i)\right) \d u\right|
\quad\leq\quad
k\;|W_i|\;
\frac{\max\{|\scale W_i|^\gamma,|\scale W_i|^\alpha\}}{n^{\frac{\gamma}{2}}}
\int_0^1
u^\gamma\d u\\
\quad=\quad
k\;
\frac{|W_i|\max\{|\scale W_i|^\gamma,|\scale W_i|^\alpha\}}{n^{\frac{\gamma}{2}}(1+\gamma)}\,,
\end{multline}
where \(k\) is the constant appearing in \eqref{eq:holder}.
Combining \eqref{eq:variance_sec_moment} and \eqref{eq:holder_lemma}, 
we deduce that the second summand has variance bounded above by
\begin{multline*}
\frac{\scale^2}{n}\sum_{i=1}^n
\Expect{k^2
\frac{|W_i|^2\max\{|\scale W_i|^{2\gamma},|\scale W_i|^{2\alpha}\}}{n^{\gamma}(1+\gamma)^2}
}
\\
\quad\leq\quad
\frac{\scale^2k^2(\scale^{2\gamma}+\scale^{2\alpha})}{n^\gamma(1+\gamma)^2}
\Expect{|W_i|^{2(1+\gamma)}+|W_i|^{2(1+\alpha)}}
\;\rightarrow\; 0\,,\;\text{ as }n\rightarrow\infty.
\end{multline*}
So the variance of the second summand vanishes asymptotically.


We turn to the expectation of the second summand.
Once again we condition implicitly on  $X_1=x_1,X_2=x_2,\dots$.
We obtain
\begin{equation*}
\Expect{\frac{\scale}{\sqrt{n}}\sum_{i=1}^nW_i\int_0^1
\left(\phi'\left(x_i+\frac{\scale u}{\sqrt{n}} W_i \right)-\phi'(x_i)\right) \d u}
\;=\;
\frac{\sum_{i=1}^nZ^{(n)}(x_i)}{n}
\,,
\end{equation*}
where $Z^{(n)}(x_i)=\scale\sqrt{n} \Expect{W_i\int_0^1\left(\phi'\left(x_i+\tfrac{\scale u}{\sqrt{n}}W_i\right)-\phi'(x_i)\right) \d u}$.
It follows from \eqref{eq:holder_lemma} that $|Z^{(n)}(x_i)|\leq\tilde{c}\,n^{\frac{1-\gamma}{2}}$, where 
$\tilde{c}=\scale
\tfrac{k}{1+\gamma}\Expect{|W_i|\max\{|\scale W_i|^\gamma,|\scale W_i|^\alpha\}}$.
We now integrate out the implicit conditioning.
The random variables $Z^{(n)}(X_1)$,\dots,$Z^{(n)}(X_n)$ are \iid, with values lying in the range
$[-\tilde{c}\,n^{\frac{1-\gamma}{2}},\tilde{c}\,n^{\frac{1-\gamma}{2}}]$.
Hence Hoeffding's inequality applies: for any positive $\eps$,
\begin{equation}\label{eq:hoeffding}
\Prob{\left|\frac{\sum_{i=1}^nZ^{(n)}(X_i)}{n}-\Expect{Z^{(n)}(X_1)}\right|>\eps}
\quad\leq\quad
2\,\exp\left(-\frac{2n^2\eps^2}{n(2\tilde{c}n^{\frac{1-\gamma}{2}})^2}\right)
\quad=\quad
2\,\exp\left(-\frac{\eps^2}{2\tilde{c}^2}n^{\gamma}\right)\,.
\end{equation}
The right-hand side of \eqref{eq:hoeffding} is summable over $n$, since \(\gamma>0\), 
and therefore the first Borel-Cantelli lemma applies:
$\tfrac{1}{n}{\sum_{i=1}^nZ^{(n)}(X_i)}$ converges almost surely to $\lim_{n\to\infty}\Expect{Z^{(n)}(X_1)}$, if such a limit exists. 

To complete the proof it suffices to show that $\lim_{n\to\infty}\Expect{Z^{(n)}(X_1)}=-\frac{\scale^2}{2}\I$.
Shifting an \(x\)-variable of integration, we achieve the following,
\begin{multline*}
\Expect{Z^{(n)}(X_1)}
\quad=\quad
\scale\sqrt{n} 
\int_\Reals
\Expect{
 W_1 \int_0^1
\left(\phi'(x+\tfrac{\scale u}{\sqrt{n}}W_1) - \phi'(x)\right)
\d u
}
e^{\phi(x)}\d x
\\
\quad=\quad
\scale\sqrt{n}
\int_0^1
\Expect{
W_1 \;
\int_\Reals
 \left( e^{\phi(x-\tfrac{\scale u}{\sqrt{n}}W_1)}
-
e^{\phi(x)}
\right)
\; 
\phi'(x)\d x
}
\d u
\\
\quad=\quad
-\scale^2
\int_0^1
\Expect{
W_1^2 \;
\int_0^1
\int_\Reals
 \left( 
 \phi'(x-\tfrac{\scale uv}{\sqrt{n}}W_1)
 e^{\phi(x-\tfrac{\scale uv}{\sqrt{n}}W_1)}
\right)
\; 
\phi'(x)\d x
\d v
}
u \d u\,.
\end{multline*}
(The exchange of integrals and expectations is justified by a Fubini argument involving the finiteness of  $\I=\int_{-\infty}^{\infty}|\phi'(x)|^2 \, f(x)\;\d x$.)
But now we undo the shift of the \(x\)-variable of integration and use the regularity condition \eqref{eq:holder} for \(\phi'\). For \(n\geq1\), this leads to:
\begin{multline*}
\left|\Expect{Z^{(n)}(X_1)} + \frac{\tau^2}{2}\I\right|
\quad=\quad
\left|\Expect{Z^{(n)}(X_1)} + 
\scale^2\int_0^1
\Expect{
W_1^2 \;
\int_0^1\int_\Reals
\phi'(x)^2e^{\phi(x)}\d x \d v
}
u \d u\right|
\\
\quad=\quad
\scale^2
\left|\int_0^1
\Expect{
W_1^2 \;
\int_0^1
\int_\Reals
 \left( 
\phi'(x) \phi'(x-\tfrac{\scale uv}{\sqrt{n}}W_1)
 e^{\phi(x-\tfrac{\scale uv}{\sqrt{n}}W_1)}
-
\phi'(x)^2e^{\phi(x)}
\right)
\;\d x
\d v
}
u \d u
\right|
\\
\quad=\quad
\scale^2
\left|\int_0^1
\Expect{
W_1^2 \;
\int_0^1
\int_\Reals
 \left( 
 \phi'(x+\tfrac{\scale uv}{\sqrt{n}}W_1)
-
\phi'(x)
\right)
\; 
\phi'(x)e^{\phi(x)}\d x
\d v
}
u \d u
\right|
\\
\quad\leq\quad
k \times 
\scale^2
\int_0^1
\int_0^1
\Expect{
W_1^2 \;
\int_\Reals
\max\left\{
\left(\frac{\scale uv}{\sqrt{n}}|W_1|\right)^\gamma
,
\left(\frac{\scale uv}{\sqrt{n}}|W_1|\right)^\alpha
\right\}
\; 
\left|
\phi'(x)\right|
e^{\phi(x)}
\d x
}
\d v \; u \d u
\\
\quad\leq\quad
k \times 
\scale^{2}(\scale^\gamma+\scale^\alpha)
\Expect{
\left(|W_1|^{2+\gamma} + |W_1|^{2+\alpha}\right) \;
}
\int_\Reals
\; 
\left|
\phi'(x)\right|
e^{\phi(x)}
\d x
\frac{1}{n^{\gamma/2}}
\quad\to\quad 0 \text{ as } n\to\infty
\,.
\end{multline*}
Here the finiteness of $\int_\Reals\left|
\phi'(x)\right|
e^{\phi(x)}
\d x=\Expect{\left|
\phi'(X_1)\right|}$
follows from $\Expect{\left|
\phi'(X_1)\right|^2}=\I<\infty$.
\end{proof}

The above result will actually be used in the following form.
\begin{cor}\label{cor:acceptance_prob}
Let $\phi:\Reals\rightarrow\Reals$, $\X=(X_1,X_2,\dots)$, $\W=(W_1,W_2,\dots)$, and \(a(\Xn,\Wn)\) be as described above in Section \ref{sec:MHRW}.
For any $N\geq 1$, 
almost surely as \(n\to\infty\) we have
$
\Expect{a(\Xn,\Wn)|\Xn,\WN}
\to
c(\tau) =
\Expect{1\wedge\exp\left(\mathcal{N}(-\frac{\scale^2}{2}\I,\scale^2 \I)\right)}
$.
\end{cor}
\begin{proof}
Given $a,b>0$,
we have $|(1\wedge ab)-(1\wedge b)|\leq|1-a|$.
This follows because if $b<1$ then $x\rightarrow 1\wedge bx$ is 1-Lipschitz, 
while if $b\geq 1$ and $a\geq\frac{1}{b}$ the left-hand side is 0,
and finally if $b\geq 1$ and $a<\frac{1}{b}$ then $a\leq ab< 1$ and $|ab-1|\leq|1-a|$. Therefore
\begin{equation*}
\left|
\Expect{\left.1\wedge\frac{f(\Xn+\frac{\scale}{\sqrt{n}} \Wn)}{f(\Xn)}-1\wedge\frac{f(\XNn+\frac{\scale}{\sqrt{n}} \WNn)}{f(\XNn)}\,\right|\,\Xn ,\WN }
\right|
\leq
\left|
\frac{f(\XN +\frac{\scale}{\sqrt{n}} \WN )}{f(\XN )} -1
\right|\,,
\end{equation*}
which converges to 0 almost surely for $n\rightarrow \infty$.
Moreover, by Lemma \ref{lemma:acceptance_prob} and the 
dominated convergence theorem, 
as \(n\to\infty\) so
$$
\Expect{\left.1\wedge\frac{f(\XNn+\frac{\scale}{\sqrt{n}} \WNn)}{f(\XNn)}\,\right|\,\Xn ,\WN }
\quad\longrightarrow\quad
\Expect{1\wedge\exp\left(\mathcal{N}(-\frac{\scale^2}{2}\I,\scale^2 \I)\right)}\,.
$$
\end{proof}

\subsection{Proving the second Mosco condition (M2)}\label{sec:mosco_2}
Suppose that the conditions of Section \ref{sec:MHRW} are satisfied.
We establish the validity of Definition \ref{defn:Mosco_convergence} (M2) before that of (M1), because (M2) follows by a more straightforward argument.
If $h\in \Hilbert\setminus \Sobolev$ then $\Phi(h)=\infty$ and thus (M2) holds trivially, for example choosing a sequence $\{h_n\}_{n=1}^\infty$ identically equal to $h$. 

Consequently we need only consider the case $h\in \Sobolev$. 
Since $\bigcup_{N\geq1}C_{0,N}^\infty(\Reals^\infty)$ is dense in $S$, there exists a sequence $\{h_k\}_{k=1}^\infty\subset \bigcup_{N\geq1}C_{0,N}^\infty(\Reals^\infty)$ such that $\|h_k-h\|_{\Sobolev}\rightarrow 0$ as \(k\to\infty\), hence $\Phi(h_k)\longrightarrow\Phi(h)$.
Choosing a subsequence and re-labelling, we may suppose that
$$
|\Phi(h_{k})-\Phi(h)|\quad\leq\quad\frac{1}{k}\qquad \text{ for }k=1,2,\ldots\,.
$$

For fixed \(k\), noting that \(h_k\in C_{0,N}^\infty(\Reals^\infty)\) for some \(N\) and that by virtue of this \(h_k\) is induced by a smooth function of compact support on \(\Reals^N\),
we see that $\Phi_n(h_k)\longrightarrow\Phi(h_k)$ as \(n\to\infty\). 
Indeed,
\begin{equation*}
\Phi_n (h_k)\quad=\quad
\Expect{\frac{\scale^2}{2}
\left(\frac{h_k(\XN +\frac{\scale }{\sqrt{n}} \WN )-h_k(\XN )}{\scale/\sqrt{n}}\right)^2
\Expect{1\wedge\frac{f(\Xn+\frac{\scale}{\sqrt{n}} \Wn)}{f(\Xn)}\,\Big|\,\XN ,\WN }
}\,.
\end{equation*}
The expression inside the outer expectation is bounded by
$\frac{\scale^2}{2}\left(|\WN | \,\|h_k'\|_\infty\right)^2$,
which is an integrable random variable.
Because of the regularity of $h_k$ and Corollary \ref{cor:acceptance_prob}, this expression converges pointwise to 
$\frac{\scale^2}{2}(\nabla h_k(\XN )^T\WN )^2\,c(\scale)$ as \(n\to\infty\).
Therefore it follows from the dominated convergence theorem that
as \(n\to\infty\) so
$\Phi_n(h_k)\longrightarrow\frac{\scale^2c(\scale)}{2}\;
\Expect{(\nabla h_k(\XN )^T\WN )^2}=\Phi(h_k)$.
Thus $|\Phi_n(h_k)-\Phi(h_k)|<\frac{1}{k}$ for sufficiently large $n$ depending on $k$,
and so we can choose an increasing sequence $j_1=1 < j_2 < \ldots$ such that for any $k=1,2,\ldots$
$$
|\Phi_n(h_{k})-\Phi(h_k)|\quad\leq\quad
\frac{1}{k}\qquad\text{ for all } n\geq j_k\,. 
$$
Note that we can in addition stipulate that $j_k\geq k$. 
For $n\geq j_1$ we define $\sigma_n=\sup\{k:j_k\leq n\}$.
Note that $1\leq\sigma_n\leq n$ and 
moreover $\sigma_n\to\infty$ as \(n\to\infty\), because $\sigma_n\geq k$ for $n\geq j_k$.
Finally, by definition of $\sigma_n$ it is the case that $j_{\sigma_n}\leq n$.
Therefore, as \(n\to\infty\),
$$
|\Phi_n(h_{\sigma_n})-\Phi(h)|
\quad\leq\quad
|\Phi_n(h_{\sigma_n})-\Phi(h_{\sigma_n})|+|\Phi(h_{\sigma_n})-\Phi(h)|
\quad\leq\quad
\frac{1}{\sigma_n}+\frac{1}{\sigma_n}
\;\longrightarrow\;0\,.
$$
It follows that as \(n\to\infty\) so \(\Phi_n(h_{\sigma_n})\to \Phi(h)\),
hence \emph{a fortiori} $\limsup_{n}\Phi_n(h_{\sigma_n})\leq \Phi(h)$.
Moreover $h_{\sigma_n}\to h$ in \(\Hilbert\),
since $h_{n}\to h$ in \(\Hilbert\) and $\sigma_n\to\infty$. 
Relabelling $h_{\sigma_n}$ as $h_n$ produces the sequence required to establish the validity of the second Mosco condition.

\subsection{Proving the first Mosco condition (M1)}\label{sec:mosco_1}
We now turn to the more substantial question of the validity of Definition \ref{defn:Mosco_convergence} (M1) under the conditions described in Section \ref{sec:MHRW}. 
Consider $h_n,$ $h\in \Hilbert$ such that $h_n\stackrel{w}\rightarrow h$ weakly in $\Hilbert$ as \(n\to\infty\).
It is convenient to write $\Phi_n (h_n)=\left\|\Psi_n (h_n)\right\|^2_{L^2_{(\X,\W,U)}}$, where
\begin{align}
\Psi_n (h_n)\quad&=\quad\sqrt{\frac{n}{2}}\left(h_n(\textbf{X}^{(n)}(1))-h_n(\textbf{X}^{(n)}(0))\right)\,.
\end{align}
Fixing $N>0$ and taking a non-zero test function $\xi$ in $C_{0}^\infty(\Reals^{2N})$ (so \(\xi\) is infinitely differentiable with compact support, and in particular is bounded),
the function $\xi(\XN,\WN)\1(U<a(\Xn,\Wn))$ belongs to $L^2_{(\X,\W,U)}$ and is also non-zero.
We can therefore apply the Cauchy-Schwartz inequality and obtain:
\begin{equation}\label{eq:cauchy_schwartz}
\sqrt{\Phi_n (h_n)}\quad=\quad
\left\|\Psi_n (h_n)\right\|_{L^2_{(\X,\W,U)}}
\quad\geq\quad
\frac{\langle\Psi_n (h_n),\xi(\XN,\WN)\1(U<a(\Xn,\Wn))\rangle_{L^2_{(\X,\W,U)}}}
{\left\|\xi(\XN,\WN)\1(U<a(\Xn,\Wn))\right\|_{L^2_{(\X,\W,U)}}}\,.
\end{equation}
Here \(U\) is the Uniform\((0,1)\) random variable introduced in Section \ref{sec:MHRW}, which is independent of \(\textbf{X}\) and \(\textbf{W}\).

Consider the denominator of \eqref{eq:cauchy_schwartz}.
Integrating out first $U$ and then $(\XNn,\WNn)$ leads to
\begin{multline}\label{eq:convergence_denominator}
\left\|\xi(\XN,\WN)\1(U<a(\Xn,\Wn))\right\|_{L^2_{(\X,\W,U)}}
\quad=\quad
\sqrt{\Expect{\xi(\XN,\WN)^2a(\Xn,\Wn)}}
\\
\;=\;
\sqrt{\Expect{\xi(\XN,\WN)^2\Expect{a(\Xn,\Wn)|\XN,\WN}}}
\quad\to\quad
\sqrt{c(\scale)}\left\|\xi(\XN,\WN)\right\|_{L^2_{(\X,\W)}}\,.
\end{multline}
Convergence as \(n\to\infty\)
follows from Corollary \ref{cor:acceptance_prob}
(hence $\Expect{a(\Xn,\Wn)|\XN,\WN}$ converges almost surely to $c(\scale)=\Expect{1\wedge\exp\left(\mathcal{N}(-\tfrac{\scale^2}{2}\I,\scale^2 \I)\right)}$)
and the fact that
$\xi(\XN,\WN)^2\Expect{a(\Xn,\Wn)|\XN,\WN}$ 
is bounded by $\|\xi\|_\infty^2<\infty$
(note that the acceptance probability \(a(\Xn,\Wn)\) lies in \([0,1]\)).

In order to deal with the numerator of \eqref{eq:cauchy_schwartz}, it is necessary to argue in more detail,
as described by the following lemma.
\begin{lem}\label{lemma:weak_convergence}
Suppose as above that \(h_n\to h\) weakly in \(\Hilbert\).
Define a \emph{twisted gradient} $\nabla^{(f)}_{\xN}\xi(\XN,\WN)$ (twisted by the density \(f\)) by requiring that it satisfy
\[
f(\XN) \nabla^{(f)}_{\xN}\xi(\XN,\WN) \quad=\quad \nabla_{\xN}\left(\xi(\XN,\WN)f(\XN)\right)\,. 
\]
Then, as \(n\to\infty\),
\begin{multline}\label{eq:weak_convergence}
\left\langle\Psi_n (h_n)\,,\;\xi(\XN,\WN)\1(U<a(\Xn,\Wn))\right\rangle_{L^2_{(\X,\W,U)}}
\\
\quad\to\quad
-\frac{\scale\,c(\scale)}{\sqrt{2}}
\Expect{h(\X)(\nabla^{(f)}_{\xN}\xi(\XN,\WN)^T\WN)}\,.
\end{multline}
\end{lem}

\begin{proof}
We use the following concise notation
\begin{align}
\x&=\xn\;,&
\xA&=\xN\;,&
\xB&=\xNn\;,&
\w&=\wn\;,&
\wA&=\wN\;,&
\wB&=\wNn\;.\nonumber
\end{align}
Fix a compact set $K\subset \Reals^{2N}$ such that $\bigcup_{n\in\mathbb{N}}\{(\xA,\wA)\,:\,\xi(\xA -\frac{\scale \wA }{\sqrt{n}},\wA )>0\}\subseteq K$.
For example, given $supp(\xi)\subseteq [-M,M]^{2N}$ (remember that $\xi$ has compact support), we can take $K=[-(1+\scale)M,(1+\scale)M]^N\times[-M,M]^N$.
Integrating out $U$ and $(X_{n+1},X_{n+2},\dots)$ the left-hand side of \eqref{eq:weak_convergence} equals
\begin{equation}\label{eq:inner_prod}
\sqrt{\frac{n}{2}} \int_{\Reals^{2n}}
\left(\tilde{h}_n(\x+\frac{\scale}{\sqrt{n}} \w)-\tilde{h}_n(\x)\right)
a(\x,\w)
\;
\xi(\xA ,\wA )
\;
f(\x)g(\w)\,\d \x\,\d \w\,,
\end{equation}
where $\tilde{h}_n(\x)=\Expect{h_n(\X)|\Xn=\x}$. 

Weak convergence of \(h_n\) to \(h\) in \(\Hilbert\) implies that \(\|h_n\|_N\leq M_1\) for some \(M_1<\infty\) by the Banach-Steinhaus theorem (the ``uniform boundedness principle'').
On the other hand, for \(b\in\Hilbert\), if \(\tilde{b}_n(\x)=\Expect{b(\X)|\Xn=\x}\) then \(\|b-\tilde{b}_n\|_{\Hilbert}\to0\) as a consequence of the \(L^2\) martingale convergence theorem.
Accordingly 
\[
|\langle \tilde{h}_n,b\rangle - \langle h_n,b\rangle|=|\langle \tilde{h}_n,\tilde{b}_n\rangle - \langle h_n,b\rangle|=|\langle h_n,\tilde{b}_n\rangle - \langle h_n,b\rangle|\leq M_1 \|b-\tilde{b}_n\|_{\Hilbert}
\quad\to\quad0\,.
\]
Thus $\tilde{h}_n\stackrel{w}\rightarrow h$ weakly in $\Hilbert$. 
These arguments show that effectively we may suppose that $h_n$ depends only on the first $n$ components, leading to $h_n(\x)=\tilde{h}_n(\x)$ for every $n$ and $\x\in\Reals^n$.



The following equality is obtained by translating $\x$ to $\x-\tfrac{\scale}{\sqrt{n}} \w$,
then multiplying and dividing through by ${f(\xB -\tfrac{\scale}{\sqrt{n}} \wB )}/{f(\xB )}$, finally using reflection to replace $\wB $ by $-\wB $ (noting that $g$ is symmetric).
\begin{multline}\label{eq:reversibility}
\int_{\Reals^{2n}}
h_n(\x+\frac{\scale}{\sqrt{n}} \w)
\left(1\wedge\frac{f(\x+\tfrac{\scale}{\sqrt{n}} \w)}{f(\x)}\right)
\;
\xi(\xA ,\wA )
\;
f(\x)g(\w)\,\d \x\,\d \w
\quad=\quad
\\
\int_{\Reals^{2n}}
h_n(\x)
\left(\frac{f(\xA )}{f(\xA -\tfrac{\scale}{\sqrt{n}} \wA )}\wedge\frac{f(\xB +\tfrac{\scale}{\sqrt{n}} \wB )}{f(\xB )}\right)
\xi(\xA -\tfrac{\scale}{\sqrt{n}} \wA ,\wA )
f(\xA -\tfrac{\scale}{\sqrt{n}} \wA )
f(\xB )g(\w)
\d \x\d \w\,.
\end{multline}
From \eqref{eq:reversibility} it follows that \eqref{eq:inner_prod} equals
\begin{multline}\label{eq:inner_prod_reversibility_0}
\sqrt{\frac{n}{2}} \int_{\Reals^{2n}}
h_n(\x)\Bigg(
\left(\frac{f(\xA )}{f(\xA -\tfrac{\scale}{\sqrt{n}} \wA )}\wedge\frac{f(\xB +\tfrac{\scale}{\sqrt{n}} \wB )}{f(\xB )}\right)
\xi(\xA -\tfrac{\scale}{\sqrt{n}} \wA ,\wA )
\;
\frac{f(\xA -\tfrac{\scale}{\sqrt{n}} \wA )}{f(\xA)}
\\
- a(\x,\w)
\;
\xi(\xA ,\wA )\Bigg)
\;
f(\x )g(\w)\,\d \x\,\d \w\,.
\end{multline}
Adding and subtracting appropriate terms to \eqref{eq:inner_prod_reversibility_0}, and multiplying and dividing the resulting second summand by $-\frac{\scale}{\sqrt{n}}$,
we obtain
\begin{multline}\label{eq:inner_prod_reversibility}
\sqrt{\frac{n}{2}}
\int_{\Reals^{2n}}
h_n(\x)
\left(\left(\frac{f(\xA )}{f(\xA -\frac{\scale}{\sqrt{n}} \wA )}\wedge\frac{f(\xB +\frac{\scale}{\sqrt{n}} \wB )}{f(\xB )}\right)
-
a(x,w)\right)\;
\xi(\xA -\tfrac{\scale}{\sqrt{n}} \wA ,\wA )
\\
\frac{f(\xA -\frac{\scale}{\sqrt{n}} \wA )}{f(\xA )}
f(\x)g(\w)\,\d \x\,\d \w\,
\\
-\frac{\scale}{\sqrt{2}} \int_{\Reals^{2n}} h_n(\x) a(\x,\w) \left(
\frac{\xi(\xA -\frac{\scale}{\sqrt{n}} \wA ,\wA )\; f(\xA -\frac{\scale}{\sqrt{n}} \wA )-
\xi(\xA ,\wA )
f(\xA )}{-\frac{\scale }{\sqrt{n}}f(\xA )}\right)
f(\x)g(\w)
\d \x\d \w
\,.
\end{multline}
Note that the density \(f\) is positive and \(C^1\) everywhere, and hence is strictly positive and bounded with bounded first derivative on the compact projection of the support of \(\xi\). 
Using Corollary \ref{cor:acceptance_prob} and smoothness and compact support of the test function $\xi$, the expression
$$
\frac{\xi(\xA -\frac{\scale}{\sqrt{n}} \wA ,\wA )f(\xA -\frac{\scale}{\sqrt{n}} \wA )-\xi(\xA ,\wA )f(\xA )}{-\frac{\scale}{\sqrt{n}}f(\xA )}\left(\int_{\Reals^{n-N}}a(\x,\w)g(\wB )\d \wB \right)
$$
converges pointwise to $(\nabla^{(f)}_{\xA }\xi(\xA ,\wA )^T\wA )c(\scale)$.
Therefore this expression is bounded by
$$
\sup\Big|\nabla_{\xA }\left(\xi(\xA ,\wA )f(\xA )\right)\Big|\times
\sup_{(\xA ,\wA )\in K} \left\{\frac{|\wA |}{f(\xA )}\right\}\times
\limsup_{n\to\infty}\int_{\Reals^{n-N}}a(\x,\w)g(\wB )\d \wB
\,,
$$
and therefore converges also in
$L^2_{(\X,\WN )}$. 
Consequently, since $h_n$ converges weakly to $h$ in $L^2_{(\X,\WN )}$ 
and the inner product of a strongly and a weakly converging sequence is a convergent sequence of real numbers (using again the uniform boundedness principle), 
the second term of \eqref{eq:inner_prod_reversibility} converges to the limit
\[
-\frac{\scale\,c(\scale)}{\sqrt{2}}
\Expect{h(\X) \left(\nabla_{\xA }^{(f)}\xi(\XN ,\WN )^T \WN \right)}\,.
\]

The proof of the lemma will be completed by showing that the first term of \eqref{eq:inner_prod_reversibility} converges to 0 as \(n\to\infty\). 
This term can be rewritten as
\begin{equation}\label{eq:first_term_zero}
\int_{\Reals^{n+N}}
b_n(\x,\wA )\,c_n(\x,\wA )
f(\x)g(\wA )\d x\,\d \wA \,,
\end{equation}
with $b_n(\x,\wA )=\frac{\scale}{\sqrt{2}}h_n(\x)\xi(\xA -\tfrac{\scale}{\sqrt{n}} \wA ,\wA )\frac{f(\xA -\tfrac{\scale}{\sqrt{n}} \wA )}{f(\xA )}$ and
\begin{multline*}
c_n(\x,\wA )
\quad=\quad
\1(\xi(\xA -\frac{\scale}{\sqrt{n}} \wA ,\wA )>0) \times\\
\times
\frac{\sqrt{n}}{\scale}
\int_{\Reals^{n-N}}
\left(
e^{\sum_{i=1}^N(\phi(x_i)-\phi(x_i-\tfrac{\scale}{\sqrt{n}} w_i))}\wedge
e^{\sum_{i=N+1}^n(\phi(x_i+\tfrac{\scale}{\sqrt{n}} w_i)-\phi(x_i ))}-
a(\x,\w)
\right)
g(\wB )\,\d \wB \,.
\end{multline*}
We shall show that $\|b_n(\x,\wA )\|_{L^2_{(\X,\WN )}}$ is bounded and $\|c_n(\x,\wA )\|_{L^2_{(\X,\WN )}}\rightarrow 0$, which implies that \eqref{eq:first_term_zero} converges to 0.

Boundedness of \(\|b_n(\x,\wA )\|_{L^2_{(\X,\WN )}}\) is almost immediate.
Since $\|h_n\|_{L^2_X}\leq M_1$ (using the uniform boundedness principle) and $\left|\xi(\xA -\tfrac{\scale}{\sqrt{n}} \wA ,\wA )\tfrac{f(\xA -\frac{\scale}{\sqrt{n}} \wA )}{f(\xA )}\right|\leq M_2$ 
(since both $\xi$ and $f$ are continuous and the set $\{(\xA,\wA)\,:\,\xi(\xA -\frac{\scale}{\sqrt{n}} \wA ,\wA )>0\}$ is contained in the compact set $K$ defined at the start of this proof),
it follows that $\|b_n(\x,\wA )\|_{L^2_{(\X,\WN )}}\leq \frac{\scale}{\sqrt{2}}M_1\,M_2$
for some positive $M_1$ and $M_2$ not depending on $n$.

Using \(f(x)=e^{\phi(x)}\), we bound the integral factor of $c_n(\x,\wA )$ as a sum of two integrals:
\begin{multline}\label{eq:justify_replace_with_Z}
\frac{\sqrt{n}}{\scale }\int_{\Reals^{n-N}}\left|
e^{\dda}\wedge e^{\db}-
1\wedge e^{\da+\db}
\right|g(\wB )\,\d \wB 
\quad\leq\quad\\
\frac{\sqrt{n}}{\scale }\int_{\Reals^{n-N}}\left|
e^{\dda}\wedge e^{\db}-
e^{\da}\wedge e^{\db}
\right|g(\wB )\,\d \wB 
+
\frac{\sqrt{n}}{\scale }\int_{\Reals^{n-N}}\left|
e^{\da}\wedge e^{\db}-
1\wedge e^{\da+\db}
\right|g(\wB )\,\d \wB \,,
\end{multline}
where 
$\da=\sum_{i=1}^N(\phi(x_i+\tfrac{\scale}{\sqrt{n}} w_i)-\phi(x_i))$, $\db=\sum_{i=N+1}^n(\phi(x_i+\tfrac{\scale}{\sqrt{n}} w_i)-\phi(x_i))$ 
and
$\dda=\sum_{i=1}^N(\phi(x_i)-\phi(x_i-\tfrac{\scale}{\sqrt{n}} w_i))$.
We deal with these two integrals separately.
Since $|(a\wedge c)- (b\wedge c)|\leq|a- b|$ for any $a,b,c>0$,
the modulus in the first integral on the right-hand side of \eqref{eq:justify_replace_with_Z} is smaller than
$\left|e^{\dda}-e^{\da}\right|$. 
Since $e^x$ is locally Lipschitz, there exist a constant $c>0$ such that, for $(\xA,\wA)\in K$, we can use \eqref{eq:holder} to deduce that
\begin{multline}\label{eq:first_comp_gradient}
\frac{\sqrt{n}}{\scale }\left|
e^{\dda}-e^{\da}
\right|
\quad\leq\quad
c\, \frac{\sqrt{n}}{\scale }\,\left|\dda-\da\right|
\\
\quad\leq\quad
c\,\sum_{i=1}^N|w_i|\int_0^1\left|\phi'(x_i+u\frac{\scale}{\sqrt{n}} w_i)-\phi'(x_i-u\frac{\scale}{\sqrt{n}} w_i)\right|\d u
\\
\quad\leq\quad
c\, k\, \frac{2^\alpha(\scale^\gamma+\scale^\alpha)\sum_{i=1}^N(|w_i|^{1+\gamma}+
|w_i|^{1+\alpha})}{n^{\gamma/2}}
\,,
\end{multline}
%
which converges to 0 uniformly over $(\xA,\wA )\in K$.

The second integral of the right-hand side of \eqref{eq:justify_replace_with_Z} can be dealt with as follows.
Suppose $\da>0$ for simplicity
(if $\da<0$ the argument needs only trivial modification).
Then
\begin{multline}\label{eq:reversibility_2}
\frac{\sqrt{n}}{\scale}\int_{\Reals^{n-1}}
\left(
e^{\da}\wedge e^{\db}
-
1\wedge e^{\da+\db}
\right)
g(\wB )\,\d \wB 
\\
\;=\;
\frac{\sqrt{n}}{\scale}\Bigg(
e^{\da}
\left(
\int_{\db>\da}g(\wB )\,\d \wB 
-
\int_{\db<-\da}e^{\db}g(\wB )\,\d \wB 
\right)
\\
\qquad -
\left(
\int_{\db>-\da}g(\wB )\,\d \wB 
-
\int_{\db<\da}e^{\db}g(\wB )\,\d \wB 
\right)
\Bigg)
\\
\;=\;
\frac{\sqrt{n}}{\scale}\Big(
\left(e^{\da}-1\right)
\left(
\int_{\db>\da}g(\wB )\,\d \wB 
-
\int_{\db<-\da}e^{\db}g(\wB )\,\d \wB 
\right)
\\
-
\int_{-\da<\db<\da}\left(1-e^{\db}\right)g(\wB )\,\d \wB 
\Big)\,.
\end{multline}
Note that $-|\da|<\db<|\da|$ implies
$\left|1-e^{\db}\right|<\left|e^{|\da|}-1\right|$
and therefore \eqref{eq:reversibility_2} is smaller in absolute value than
\begin{equation}\label{eq:reversibility_3}
\frac{\sqrt{n}}{\scale}\left|e^{|\da|}-1\right|
\left(
\left|
\int_{\db>\da}g(\wB )\,\d \wB 
-
\int_{\db<-\da}e^{\db}g(\wB )\,\d \wB 
\right|
+
\int_{-|\da|<\db<|\da|}g(\wB )\,\d \wB 
\right)\,.
\end{equation}
To complete the proof of the lemma, we show that \eqref{eq:reversibility_3} is bounded for
$(\xA,\wA )\in K$ and converges almost surely to 0 as \(n\to\infty\).
The integral terms of \eqref{eq:reversibility_3} are bounded either by 1 or by the (finite) supremum of $e^{-\da}$ over
$(\xA,\wA )\in K$.
Moreover, since the function $x\rightarrow e^x$ is locally Lipschitz,
 there exist $c>0$ such that for $(\xA,\wA )\in K$
$$
\frac{\sqrt{n}}{\scale}\left|e^{|\da|}-1\right|
\quad\leq\quad
\frac{\sqrt{n}}{\scale}c\left|\da\right|
\quad\leq\quad
c\,\sum_{i=1}^N\int_0^1|w_i|\left|\phi'(x_i+u\frac{\scale w_i}{\sqrt{n}})\right|\d u\;,
$$
which is bounded over $(\xA,\wA )\in K$.
Therefore \eqref{eq:reversibility_3} is bounded.
Finally, for almost every $\wA $ and $x_1,x_2,\dots$ it is the case that
$\da$ converges to 0 and $\db\stackrel{\mathcal{D}}\rightarrow \mathcal{N}(-\frac{\scale^2}{2}\I,\scale^2 \I)$ (see Lemma \ref{lemma:acceptance_prob}).
Therefore the integral $\int_{-\da<\db<\da}g(\wB )\,\d \wB $ converges almost surely to 0 and
$$
\int_{\db>\da}g(\wB )\,\d \wB -
\int_{\db<-\da}e^{\db}g(\wB )\,\d \wB 
$$
converges almost surely to 
$$
\int_{z>0}\exp\left\{-\frac{(z+\frac{\scale^2}{2}\I)^2}{2\scale^2 \I}\right\}\d z
-
\int_{z<0}e^z\exp\left\{-\frac{(z+\frac{\scale^2}{2}\I)^2}{2\scale^2 \I}\right\}\d z
\quad=\quad 0\,.
$$
Thus the second integral of the right-hand side of \eqref{eq:justify_replace_with_Z} converges to \(0\) as \(n\to\infty\).
Accordingly we have shown that the first term of \eqref{eq:inner_prod_reversibility} converges to 0 as \(n\to\infty\),
and so this completes the proof of the lemma.
\end{proof}

From \eqref{eq:cauchy_schwartz}, \eqref{eq:convergence_denominator} and  Lemma \ref{lemma:weak_convergence} it follows that for any $\xi\in C_0^{\infty}(\Reals^{2N})$ with $\xi\neq 0$
\begin{equation}\label{eq:mosco_1_weak_inequality}
\liminf_{n\rightarrow\infty}\sqrt{\Phi_n(h_n)}
\;\geq\;
-\frac{\scale\,\sqrt{c(\scale)}}{\sqrt{2}}
\frac{\Expect{h(\X)\left(\nabla^{(f)}_{\xN}\xi(\XN,\WN)^T\WN\right)}}
{\left\|\xi(\XN,\WN)\right\|_{L^2_{(X,W)}}}\;.
\end{equation}
Given \eqref{eq:mosco_1_weak_inequality}, we can prove (M1) of Definition \ref{defn:Mosco_convergence} using Hilbert space duality.
We consider $h\in \Sobolev$ and then $h\in \Hilbert\setminus \Sobolev$.
If $h\in \Sobolev$, then an integration-by-parts argument using the compact support of $\xi$ shows that
$$
-\Expect{h(\X)\left(\nabla^{(f)}_{\xN}\xi(\XN,\WN)^T\WN\right)}
\quad=\quad
\Expect{\xi(\XN,\WN)\left(\nabla_{\xN} h(\X)^T\WN\right)}\,.
$$
Since $\xi$ depends on $(\XN,\WN)$ only and $\Expect{\partial_i h(\X) W_i}=\Expect{\partial_i h(\X)}\Expect{W_i}=0$, we find
$$
\Expect{\xi(\XN,\WN)\left(\nabla_{\xN} h(\X)^T\WN\right)}
\quad=\quad
\Expect{\xi(\XN,\WN)\left(\nabla h(\X)^T\W\right)}\,.
$$
Using Hilbert space duality and taking the supremum over $N$ and $\xi$ we obtain the desired inequality
\begin{multline*}
\liminf_{n\rightarrow\infty}\sqrt{\Phi_n(h_n)}
\quad\geq\quad
\sup_{N\geq1}
\sup_{\substack{\xi\in C_0^{\infty}(\Reals^{2N})\\\xi\neq 0}}
\frac{\scale\,\sqrt{c(\scale)}}{\sqrt{2}}
\frac{\Expect{\xi(\XN ,\WN )\left(\nabla h(\X)^T\W\right) }}
{\left\|\xi(\XN ,\WN )\right\|_{L^2_{(\X,\W)}}}
\\
\quad=\quad
\frac{\scale\,\sqrt{c(\scale)}}{\sqrt{2}}
\left\|\nabla h(\X)^T\W\right\|_{L^2_{(\X,\W)}}
\quad=\quad
\frac{\scale\,\sqrt{c(\scale)}}{\sqrt{2}}
\sqrt{\Expect{|\nabla h(\X)|^2}}
\quad=\quad
\sqrt{\Phi(h)}\,.
\end{multline*}
This establishes (M1) of Definition \ref{defn:Mosco_convergence} for the case of \(h\in \Sobolev\).

On the other hand, (M1) follows for the case of $h\in \Hilbert\setminus \Sobolev$ 
if it can then be shown that the supremum over $\xi$ of the right-hand side of \eqref{eq:mosco_1_weak_inequality} is equal to infinity.
Since $h\notin \Sobolev$, we can use Hilbert space duality, together with the definition of \(\Sobolev\), and also the definition of the twisted gradient in Lemma \ref{lemma:weak_convergence}, to show
that
\begin{equation}\label{eq:diverging_case}
\sup_{N\geq 1}\;
\sup_{\substack{\xi_1\in C_0^{\infty}(\Reals^N)\\\xi_1\neq 0}}
\frac{\langle\,h\,,\,\xi_1\,\rangle_{\Hilbert}+\langle\, h\,,\,-\sum_{i=1}^N\nabla_i^{(f)}\nabla_i\xi_1\,\rangle_{\Hilbert}}{\left\|\xi_1\right\|_{\Sobolev}}
\;=\;
\infty\,.
\end{equation}
(For otherwise the numerator, viewed as a function of \(\xi\), extends to a continuous linear function on \(\Sobolev\), and the Riesz representation theorem for Hilbert space would then imply that \(h\in\Sobolev\).)
Since $h\in \Hilbert$ and therefore $\frac{\langle\,h\,,\,\xi_1\,\rangle_{\Hilbert}}{\left\|\xi_1\right\|_{\Sobolev}}\leq \frac{\|h\|_{\Hilbert}\,\|\xi_1\|_{\Hilbert}}{\left\|\xi_1\right\|_{\Sobolev}}\leq\|h\|_{\Hilbert}<\infty$,
it follows from \eqref{eq:diverging_case} that
\begin{equation}\label{eq:duality}
\sup_{N\geq 1}\;
\sup_{\substack{\xi_1\in C_0^{\infty}(\Reals^N)\\\xi\neq 0}}
\frac{\langle\, h\,,\,-\sum_{i=1}^N\nabla_i^{(f)}\nabla_i\xi_1\,\rangle_{\Hilbert}}{\left\|\xi_1\right\|_{\Sobolev}}
\;=\;
\infty\,.
\end{equation}
To apply \eqref{eq:duality} to \eqref{eq:mosco_1_weak_inequality}, we consider test functions $\xi$ of the form
$\xi(\XN ,\WN )=\sum_{i=1}^N\xi_2(W_i)\nabla_i\xi_1 (\X)$,
with $\xi_1 $ in $C_0^{\infty}(\Reals^N)$ and $\xi_2$ in $C_0^{\infty}(\Reals)$, choosing \(\xi_2\) so that (for all indices \(i\)) $\Expect{\xi_2(W_i)}=\Expect{\xi_2(W_1)}=0$.
For such a test function
\begin{equation}\label{eq:denominator_duality}
\|\xi(\XN ,\WN )\|^2_{\lxw}
\quad=\quad
\Expect{\sum_{i=1}^N\xi_2(W_i)^2\nabla_i\xi_1 (\X)^2}
\quad=\quad
\|\xi_2\|^2_{L^2_{W_{1}}}{\sum_{i=1}^N\|\nabla_i\xi_1 \|^2_{\Hilbert}}\;.
\end{equation}
Moreover, since  $\Expect{\xi_2(W_i)}=0$ for all indices \(i\), we have
\begin{multline}\label{eq:numerator_duality}
\Expect{h(\X)\sum_{j=1}^N\sum_{i=1}^N\xi_2(W_i)W_j\nabla_j^{(f)}\nabla_i\xi_1 (\X)}
\;=\;
\Expect{h(\X)\sum_{j=1}^N\xi_2(W_j)W_j\nabla_j^{(f)}\nabla_j\xi_1 (\X)}
\;=\\
\Expect{h(\X)\sum_{j=1}^N\nabla_j^{(f)}\nabla_j\xi_1 (\X)}
\Expect{\xi_2(W_1)W_1}
\end{multline}
Combining \eqref{eq:denominator_duality} and \eqref{eq:numerator_duality},
and using the specific form of the test function \(\xi\),
the supremum of the right-hand side of \eqref{eq:mosco_1_weak_inequality} is controlled by a fixed positive finite multiple of 
\begin{equation}\label{eq:diverging}
\left(\sup_{\substack{\xi_2\in C_0^{\infty}(\Reals)\\\xi_2\neq 0,\,\Expect{\xi_2(W_1)}=0}}
\frac{\Expect{\xi_2(W_1)W_1}}
{\|\xi_2\|_{L^2_{W_{1}}}}
\right)
\times
\left(
\sup_{N\geq 1}\;\sup_{\substack{\xi_1 \in C_0^{\infty}(\Reals^N),\\\xi_1 \neq 0}}
\frac{\Expect{-h(\X)\sum_{j=1}^N\nabla_j^{(f)}\nabla_j\xi_1 (\X)}}
{\sqrt{\sum_{i=1}^N\|\nabla_i\xi_1 \|^2_{\Hilbert}}}
\right)
\,.
\end{equation}
Now $W_1$ can be arbitrarily approximated in $L^2_{W_1}$ by mollifications $\xi_2(W_1)$ such that $\xi_2\in C_0^{\infty}(\Reals)$ and $\Expect{\xi_2(W_1)}=0$.
Consequently
the supremum over $\xi_2$ in \eqref{eq:diverging} is equal to $\Expect{W_1^2}=1$.
Therefore \eqref{eq:diverging} equals
\begin{equation*}
\sup_{\substack{\xi_1 \in C_0^{\infty}(\Reals^N),\\\xi_1 \neq 0}}
\frac{\Expect{-h(\X)\sum_{j=1}^N\nabla_j^{(f)}\nabla_j\xi_1 (\X)}}
{\sqrt{\sum_{i=1}^N\|\nabla_i\xi_1 \|^2_{\Hilbert}}}
\quad\geq\quad
\sup_{\substack{\xi_1 \in C_0^{\infty}(\Reals^N),\\\xi_1 \neq 0}}
\frac{\Expect{-h(\X)\sum_{j=1}^N\nabla_j^{(f)}\nabla_j\xi_1 (\X)}}
{\|\xi_1\|_{\Sobolev}}
\;=\;\infty\,,
\end{equation*}
where the infinite value of the second supremum follows from \eqref{eq:duality}.

This establishes (M1) of Definition \ref{defn:Mosco_convergence} for the case of \(h\in \Hilbert\setminus \Sobolev\), and thus (M1) holds for all \(h\in \Sobolev\).

The results of this section and of Section \ref{sec:mosco_2} therefore together establish Mosco convergence of \(\Phi_n\) to \(\Phi\).

\section{Weak convergence}\label{sec:weak}
In this section we show that a strengthening of \eqref{eq:holder} to deliver a global Lipschitz property for \(\phi'\)
permits control of the \(\Phi_n\) by the Sobolev norm associated with \(\Phi\).
This suffices to allow the application of the results of \cite{Sun-1998} to establish (probabilistic) weak convergence.

\begin{lem}\label{lemma:capacity_bound}
Suppose that $\phi'$ is Lipschitz-continuous, meaning that $|\phi'(x+v)-\phi'(x)|
<k|v|$ for a fixed $k$ and for all $x,v\in \Reals$.
Then there exists $C$ depending on \(\tau\) but not depending on $n$ such that, for any $h\in \Hilbert$,
\begin{equation}\label{eq:capacity_bound}
\Phi_n (h)
\quad\leq\quad
C\;\left(\|h\|_{\Hilbert}^2+\Phi(h)\right)\,.
\end{equation}
\end{lem}
\begin{proof}
If $\Phi(h)=\infty$, then \eqref{eq:capacity_bound} holds trivially (note that $\Phi_n (h)<\infty$ whenever $h\in \Hilbert$).
We may therefore suppose that $\Phi(h)<\infty$.

\newcommand{\Division}{\sum_{i=1}^n|W_i|^2}
Viewing \(\Phi_n(h)\) as an expectation as in Equation \eqref{eq:def_of_phi_n}, 
we divide the expectation according to whether or not \(\Division\) is greater than
$\cn$ for a suitable constant \(\cn\).
\begin{multline*}
 \Phi_n(h)\quad=\quad \frac{n}{2}\;\Expect{\left(h(\textbf{X}^{(n)}(1))-h(\textbf{X}^{(n)}(0))\right)^2}\\
\quad=\quad
\frac{n}{2}\,\Expect{\left(h(\Xn+\tfrac{\scale}{\sqrt{n}} \Wn,\Xtail) - h(\Xn,\Xtail)\right)^2 a(\Xn,\Wn)\;;\; \Division \leq c_n}\\
 +
 \frac{n}{2}\,\Expect{\left(h(\Xn+\tfrac{\scale}{\sqrt{n}} \Wn,\Xtail) - h(\Xn,\Xtail)\right)^2 a(\Xn,\Wn)\;;\; \Division > c_n}\,.
\end{multline*}
We focus first on the component for which \(\Division> \cn\).
Observe that
\begin{multline*}
\left(h(\xn+\tfrac{\scale}{\sqrt{n}} \wn,\xtail)-h(\xn,\xtail)\right)^2
(f(\xn)\wedge f(\xn+\tfrac{\scale}{\sqrt{n}} \wn))
\\
\quad\leq\quad
2\left(h(\xn+\tfrac{\scale}{\sqrt{n}} \wn,\xtail)^2 + h(\xn,\xtail)^2\right)
(f(\xn)\wedge f(\xn+\tfrac{\scale}{\sqrt{n}} \wn))
\\
\quad\leq\quad
2\left(h(\xn+\tfrac{\scale}{\sqrt{n}} \wn,\xtail)^2 f(\xn+\tfrac{\scale}{\sqrt{n}} \wn) + h(\xn,\xtail)^2f(\xn)\right)\,.
\end{multline*}
Changing variables $\xn\rightarrow \xn-\frac{\scale \wn}{\sqrt{n}}$ in the integral expression of the above, we may deduce that
\begin{multline}\label{eq:first-part-of-scaling}
\frac{n}{2}\,\Expect{\left(h(\Xn+\tfrac{\scale}{\sqrt{n}} \Wn,\Xtail) - h(\Xn,\Xtail)\right)^2 a(\Xn,\Wn)\;;\; \Division > \cn}\\
\quad\leq\quad 
  2 n \, \Expect{h(\Xn,\Xtail)^2\;;\;\Division > \cn}
  \quad=\quad
  2\Expect{h(\X)^2} \; n \, \Prob{\Division > c_n} \,.
\end{multline}
Now consider the Chernoff bound for the \(\chi^2\) distribution.
When \(\cn>n\) we have
\[
 \Prob{\Division > \cn} \quad\leq\quad  \left(\tfrac{\cn}{n} e^{-\left(\tfrac{\cn}{n}-1\right)}\right)^{n/2}\,,
\]
and so the upper bound in \eqref{eq:first-part-of-scaling} converges to zero if (for example) \(\cn = (1+\varepsilon)n\) for some \(\varepsilon>0\).

Now consider the component for which \(\Division\leq \cn\).
Employing Jensen's inequality, and changing measure by translation,
\begin{multline}\label{eq:second_part}
\frac{n}{2}\,\Expect{\left(h(\Xn+\tfrac{\scale}{\sqrt{n}} \Wn,\Xtail) - h(\Xn,\Xtail)\right)^2 a(\Xn,\Wn)\;;\; \Division \leq \cn}\\
\quad=\quad 
\frac{n}{2}\,
\mathbb{E}\Big[\left(
\frac{\scale}{\sqrt{n}}\int_0^1
\left\langle
 \nabla_{1:n} h(\Xn+\tfrac{\scale}{\sqrt{n}}u \Wn,\Xtail)^T \Wn
\right\rangle
\d u
\right)^2  
\\
\qquad\qquad\times\frac{f(\Xn)\wedge f(\Xn+\frac{\scale}{\sqrt{n}}\Wn)}{f(\Xn)}
\;;\; \Division  \leq \cn
\Big]
\\
 \quad\leq\quad 
  \frac{\scale^2}{2}\,
\mathbb{E}\Big[
 \int_0^1
 \left\langle
  \nabla_{1:n} h(\Xn+\tfrac{\scale}{\sqrt{n}}u \Wn,\Xtail)^T \Wn
 \right\rangle^2\\
 \qquad\qquad\qquad\times
 \frac{f(\Xn)\wedge f(\Xn+\frac{\scale}{\sqrt{n}}\Wn)}{f(\Xn)} \d u
 \;;\;\Division \leq \cn
 \Big]
\\
\quad=\quad 
\frac{\scale^2}{2}\,\Expect{
 \left\langle
  \nabla_{1:n} h(\Xn,\Xtail)^T \Wn
 \right\rangle^2
 \int_0^1
 \Lambda_n(u;\Xn,\Wn)
 \d u
 \;;\; 
 \Division \leq \cn}
\,, 
\end{multline}
using \(\Lambda_n(u;\Xn,\Wn)=\exp((\sum_{i=1}^n\phi(X_i-\tfrac{\scale}{\sqrt{n}}u W_i)) \wedge (\sum_{i=1}^n\phi(X_i+\tfrac{\scale}{\sqrt{n}}(1-u)W_i)) - \sum_{i=1}^n\phi(X_i))\).
Now observe that if \(0\leq u \leq 1\) then
\begin{multline}
\log(\Lambda_n(u;\Xn,\Wn))=
\left(\sum_{i=1}^n\left(\phi(X_i-\tfrac{\scale}{\sqrt{n}}u W_i)-\phi(X_i)\right)\right)
\wedge
\left(\sum_{i=1}^n\left(\phi(X_i+\tfrac{\scale}{\sqrt{n}}(1-u)W_i)-\phi(X_i)\right)\right)
=\\
\left(-u \,\tfrac{\scale}{\sqrt{n}}\sum_{i=1}^n W_i\phi'(X_i)-\tfrac{\scale}{\sqrt{n}}u\sum_{i=1}^n W_i\int_0^1\phi'(X_i-(1-s)\tfrac{\scale}{\sqrt{n}}u W_i)-\phi'(X_i)ds\right)
\wedge\\
\left((1-u)\,\tfrac{\scale}{\sqrt{n}}\sum_{i=1}^nW_i\phi'(X_i)+\tfrac{\scale}{\sqrt{n}}(1-u)\sum_{i=1}^nW_i\int_0^1\phi'(X_i+s\tfrac{\scale}{\sqrt{n}}(1-u)W_i)-\phi'(X_i)ds\right)
\leq\\
\left(-\tfrac{\scale}{\sqrt{n}}u\sum_{i=1}^n W_i\int_0^1\phi'(X_i-(1-s)\tfrac{\scale}{\sqrt{n}}u W_i)-\phi'(X_i)ds)\right)
\vee\\
\qquad\qquad\left(\tfrac{\scale}{\sqrt{n}}(1-u)\sum_{i=1}^nW_i\int_0^1\phi'(X_i+s\tfrac{\scale}{\sqrt{n}}(1-u)W_i)-\phi'(X_i)ds\right)
\\
\quad\leq\quad
\tfrac{\scale}{\sqrt{n}}\sum_{i=1}^n k\tfrac{\scale}{\sqrt{n}}|W_i|^2
\quad=\quad
k\,\tfrac{\scale^2}{n}\sum_{i=1}^n |W_i|^2\,.
\end{multline}
Note that the terms involving \(\tfrac{\scale}{\sqrt{n}}\sum_{i=1}^n W_i\phi'(X_i)\) can be removed because of the following reasoning:
if \(A\geq 0\) then
\((-uA+B)\wedge((1-u)A + C)\leq -uA+B\ \leq B\),
while if \(A<0\) then
\((-uA+B)\wedge((1-u)A + C)\leq (1-u)A+C\ \leq C\).
Thus \((-uA+B)\wedge((1-u)A + C)\leq B\vee C\).

Therefore the right-hand side of \eqref{eq:second_part} is itself bounded as follows:
\begin{multline}\label{eq:bound_on_second_part}
\frac{\scale^2}{2}\,\mathbb{E}\Big[
\left\langle
  \nabla_{1:n} h(\X)^T \Wn
 \right\rangle^2
\int_0^1
\Lambda_n(u;\Xn,\Wn)
\d u
\;;\; \Division \leq \cn
\Big]
\;\leq\\
\frac{\scale^2}{2}\,\mathbb{E}\Big[
\left\langle
   \nabla_{1:n} h(\X)^T \Wn
 \right\rangle^2
\exp\left(k\,\tfrac{\scale^2}{n}c_n\right)
\Big]
\;\leq\;
\frac{\scale^2}{2}\exp\left(\scale^2 k(1+\epsilon)\right)\,\mathbb{E}\Big[
\left\|\nabla_{1:n} h(\X)\right\|^2
\Big]
\,.
\end{multline}
Combining \eqref{eq:first-part-of-scaling} and \eqref{eq:bound_on_second_part} we have
\begin{equation*}
\Phi_n(h)
\quad\leq\quad
\sup_{n\geq 1}\left(2 \; n \, \Prob{\Division > c_n}\right) \Expect{h(\X)^2}
+
\left(\frac{\scale^2}{2}\exp\left(\scale^2 k(1+\epsilon)\right)\right)\,
\Expect{\left\|\nabla h(\X)\right\|^2}\,.
\end{equation*}
The desired result now follows because $n \Prob{\Division > c_n}$ converges to 0 as $n\to\infty$.
 \end{proof}
 
We may now apply \citet[Theorem 1]{Sun-1998} to deduce weak convergence of $\{\textbf{X}^{(n)}(t): t\geq0\}$  to $\{\textbf{X}^\infty(t)\;:\;t\geq0\}$
as described in Theorem \ref{thm:weak_conv_of_dir_forms} in Section \ref{sec:results} above.
\section{Discussion}\label{sec:discussion}
%
The above work demonstrates that Dirichlet forms provide an effective methodology for treating the Optimal Scaling framework in its natural infinite-dimensional context,
and also for reducing the framework's dependence on severe regularity conditions.
It is interesting to compare the Dirichlet form approach with that of the recent paper by \cite{DurmusLeCorffMoulinesRoberts-2016},
which does manage to reduce the regularity conditions required by the classical \cite{RobertsGelmanGilks-1997} approach (though not to the same extent as above),
and also substantially relaxes smoothness requirements. 
It would be interesting to see whether the smoothness requirements of the Dirichlet form approach could be similarly reduced.

In this paper we have focussed on establishing the utility of the Dirichlet form approach for the special case of \iid~targets and for the Metropolis-Hastings random walk sampler; 
we expect this approach will prove useful in studying optimal scaling for MALA, 
and for non-identically distributed targets \citep{Bedard-2006a},
and for the non-independent case \citep{BreyerRoberts-2000,MattinglyPillaiStuart-2012}.
Tied as it is to equilibrium calculations, it is less clear how to extend the approach of this paper to deal with the transient behaviour of MCMC algorithms \emph{before} they reach equilibrium
(see for example the results of \citealp{ChristensenRobertsRosenthal-2005,JourdainLelievreMiasojedow-2014,KuntzOttobreStuart-2016a}), and this is a clear challenge for future work.  
Finally, there is evidently scope for adapting the Dirichlet form approach to deal with Optimal Scaling frameworks in which there is a natural Banach-space structure,
and in this case we expect that the genuinely infinite-dimensional nature of the Dirichlet form approach will be highly beneficial.
The techniques discussed here (especially that of Mosco convergence) also seem to have considerable potential for other high- or infinite-dimensional problems in applied probability.


\appendix
\section{Existence of the limiting infinite-dimensional stochastic process}\label{appendix:capacity}
This appendix is devoted to proving the existence of an infinite-dimensional Markov process associated to the limiting Dirichlet form $\Phi$ defined by Equation \eqref{eq:def_of_phi}.
\citet{AlbeverioRoeckner-1989} consider Dirichlet forms of this kind (sometimes called \emph{classic} Dirichlet forms) in the framework of topological vector spaces (which includes our case).
They provide and discuss a sufficient set of four conditions \citep[(2.8)-(2.11)]{AlbeverioRoeckner-1989}
(which we refer to below as conditions \AR1-4 respectively)
for the existence of a diffusion process associated to $\Phi$ \citep[Thm.2.7]{AlbeverioRoeckner-1989}.
In summary, the conditions \AR1-4 imply that $\Phi$ is a (local) quasi-regular Dirichlet form \citep[Definition 3.3.1]{MaRoeckner-1992}, 
and this in turn implies the existence of an associated Markov process \citep[Theorem 3.5]{MaRoeckner-1992}.

Therefore in this section we only need to show that the conditions \AR1-4 are satisfied.
In our case, the only non-trivial task is to prove \AR1.
Indeed, since the state space $(\Reals^\infty,\pi^{\otimes\infty})$ is both a Fr\'echet space and a Polish space (\citealp[Chapter IV]{Conway-1994}; \citealp[Exercises 3.1-3.2]{Eldredge-2012}), 
the conditions \AR2,4 
follow respectively from Remark 2.4.(i) and Proposition 2.6 of \citet{AlbeverioRoeckner-1989}.
Moreover condition \AR3 
requires that if $h_1,h_2\in D(\Phi)=\Sobolev$ are continuous and have disjoint supports 
($\support(h_1)\cap \support(h_2)=\emptyset$) then $\Phi(h_1,h_2)=0$.
In our case $\Phi(h_1,h_2)=\Expect{\langle\nabla h_1(\X),\nabla h_2(\X)\rangle_{\ell^2}}$ and $\support(h_1)\cap \support(h_2)=\emptyset$ implies $\langle\nabla h_1(\X),\nabla h_2(\X)\rangle_{\ell^2}=0$ almost surely, 
so \AR3 
follows.

We conclude by undertaking the only non-trivial task, that of verifying condition \AR1.
Given Definition \ref{def:capacity}, condition \AR1 is satisfied if we exhibit
a \(\Phi\)-nest of compact sets:
\begin{equation}\label{eq:condition}
\text{There exist compact sets }
K^{(1)}\subseteq K^{(2)} \subseteq K^{(3)} \subseteq \ldots \subseteq\Reals^\infty
\text{ such that }
\lim_{n\rightarrow\infty} \Capacity(\Reals^\infty\setminus K^{(n)})=0
\,.
\end{equation}
We now prove that Condition \eqref{eq:condition} holds for the Dirichlet form $\Phi$ defined by Equation \eqref{eq:def_of_phi}.
We will use
\begin{equation}\label{eq:defi:K_n}
K^{(n)} \quad=\quad \CartesianProduct_{\ell=1}^\infty \left[-2k^{(n)}_\ell,2k^{(n)}_\ell\right]\,, 
\end{equation}
where, for any positive integers \(n\) and \(\ell\),
\begin{equation*}
k^{(n)}_\ell
\quad=\quad
(n\,\ell) \wedge
\left(
\inf\left\{x\geq 0\;:\;\pi([-x,x])\geq \exp\left(-\frac{1}{n\,\ell^2}\right)\right\}
\right)
\,.
\end{equation*}
Note that it is the case that $0<k^{(n)}_\ell<\infty$ for any positive integers $n$ and $i$.
Since cartesian products of compact sets are compact in the product topology (Tychonoff's theorem)
it follows that
the set $K^{(n)}$ is a compact subset of $\Reals^{\infty}$.

The following lemma completes the proof of \eqref{eq:condition}.
\begin{lem}
Given $K^{(n)}$ as in \eqref{eq:defi:K_n} it is the case that $\Capacity(\Reals^\infty\setminus K^{(n)})\to0$ as \(n\to\infty\).
\end{lem}
\begin{proof}
For positive integers $i$ and $n$, we define the function $b^{(n)}_\ell:\Reals\rightarrow(\Reals^+\cup\{\infty\})$ piece-wise by 
\begin{equation}\label{eq:defi_l}
b^{(n)}_\ell(x_\ell)
\quad=\quad
\begin{cases}
0 & \text{ for }|x_\ell|<k^{(n)}_\ell\,,
\\
\frac{x_\ell-k^{(n)}_\ell}{2k^{(n)}_\ell-x_\ell} & \text{ for } k^{(n)}_\ell\leq|x_\ell|\leq 2k^{(n)}_\ell\,,
\\
\infty & \text{ for }  2k^{(n)}_\ell<|x_\ell|\,,
\end{cases}
\end{equation} 
and the function \(u^{(n)}\) (defined for \(x\in\Reals^\infty\)) by
\begin{equation}\label{eq:defi_u}
u^{(n)}(x)
\quad=\quad
\begin{cases}
\frac{\sum_{\ell=1}^\infty b^{(n)}_\ell(x_\ell)}{1+\sum_{\ell=1}^\infty b^{(n)}_\ell(x_\ell)} & \text{ for } \sum_{\ell=1}^\infty b^{(n)}_\ell(x_\ell)<\infty\,,
\\
1 & \text{ for } \sum_{\ell=1}^\infty b^{(n)}_\ell(x_\ell)=\infty\,.
\end{cases}
\end{equation} 
Note that if $x\in\Reals^\infty\setminus K^{(n)}$ then $b^{(n)}_\ell(x_\ell)=\infty$ for at least one $i$ in $\Numbers$ and therefore $u^{(n)}=1$ on $\Reals^\infty\setminus K^{(n)}$.
Consequently $\Capacity(\Reals^\infty\setminus K^{(n)})\leq \|u^{(n)}\|^2_{\Hilbert} + \Phi(u^{(n)})$.

So the lemma is proved if we can show that $\|u^{(n)}\|^2_{\Hilbert}\rightarrow0$ and $\Phi(u^{(n)})\rightarrow0$.

We begin by considering $\|u^{(n)}\|^2_{\Hilbert}$.
Since $0\leq u^{(n)}\leq 1$ and $u^{(n)}(x)=0$ for $x\in \CartesianProduct_{\ell=1}^\infty \left[-k^{(n)}_\ell,k^{(n)}_\ell\right]$, it is the case that as \(n\to\infty\) so
\begin{multline*}
\|u^{(n)}\|^2_{\Hilbert}
\;\leq\;
\int_{\Reals^\infty\setminus\CartesianProduct_{\ell=1}^\infty \left[-k^{(n)}_\ell,k^{(n)}_\ell\right]}\,\pi^{\otimes\infty}\left(dx\right)
\;=\;
1-\pi^{\otimes\infty}\left(\CartesianProduct_{\ell=1}^\infty \left[-k^{(n)}_\ell,k^{(n)}_\ell\right]\right)
\;=\;
1-\prod_{\ell=1}^\infty\pi\left(\left[-k^{(n)}_\ell,k^{(n)}_\ell\right]\right)
\\
\;\leq\;
1-\prod_{\ell=1}^\infty\exp\left(-\frac{1}{n\, \ell^2}\right)
\;=\;
1-\exp\left(-\frac{1}{n}\sum_{\ell=1}^\infty\frac{1}{\ell^2}\right)
\;=\;
1-\exp\left(-\frac{\pi^2}{6n}\right)
\;\longrightarrow\;
0\,.
\end{multline*}

We turn to consideration of $\Phi(u^{(n)})$.
From \eqref{eq:def_of_phi} we know $\Phi(u^{(n)})=\frac{\scale c(\scale)}{2}\sum_{\ell=1}^\infty\Expect{\left|\frac{\partial u^{(n)}}{\partial x_\ell}(\X)\right|^2}$.
From \eqref{eq:defi_l} and \eqref{eq:defi_u} it follows that if \(x\in\Reals^\infty\) then
\begin{equation}\label{eq:u_prime}
\frac{\partial u^{(n)}}{\partial x_\ell}(x)
\quad=\quad
\begin{cases}
0 
& \text{ for } |x_\ell|< k^{(n)}_\ell\text{ or }  |x_\ell|> 2k^{(n)}_\ell\,,
\\
\frac{\partial }{\partial x_\ell}\left(\frac{\sum_{j\neq \ell}^\infty b^{(n)}_j(x_j)+b^{(n)}_\ell(x_\ell)}{1+\sum_{j\neq \ell}^\infty b^{(n)}_j(x_j)+b^{(n)}_\ell(x_\ell)}\right)
& \text{ for }k^{(n)}_\ell\leq|x_\ell|\leq 2k^{(n)}_\ell\text{ and }\sum_{j\neq \ell}^\infty b^{(n)}_j(x_j)<\infty\,,
\\
0
& \text{ for } k^{(n)}_\ell\leq|x_\ell|\leq 2k^{(n)}_\ell\text{ and }\sum_{j\neq \ell}^\infty b^{(n)}_j(x_j)=\infty\,.
\end{cases}
\end{equation}
For $k^{(n)}_\ell\leq|x_\ell|\leq 2k^{(n)}_\ell$ and $\sum_{j\neq \ell}^\infty b^{(n)}_j(x_j)<\infty$ it is the case that
\begin{multline}\label{eq:ineq_u_prime}
\left|\frac{\partial }{\partial x_\ell}\left(\frac{\sum_{j\neq \ell}^\infty b^{(n)}_j(x_j)+b^{(n)}_\ell(x_\ell)}{1+\sum_{j\neq \ell}^\infty b^{(n)}_j(x_j)+b^{(n)}_\ell(x_\ell)}\right)\right|
\;=\;
\left|\frac{\frac{\partial }{\partial x_\ell}b^{(n)}_\ell(x_\ell)}{
\left(1+\sum_{j\neq \ell}^\infty b^{(n)}_j(x_j)+b^{(n)}_\ell(x_\ell)\right)^2}\right|
\;\leq\;
\left|\frac{\frac{\partial }{\partial x_\ell}b^{(n)}_\ell(x_\ell)}{
\left(1+b^{(n)}_\ell(x_\ell)\right)^2}\right|\\
\;=\;
\left|\frac{\partial }{\partial x_\ell}\left(\frac{b^{(n)}_\ell(x_\ell)}{1+b^{(n)}_\ell(x_\ell)}\right)\right|
\;=\;
\left|\frac{\partial }{\partial x_\ell}\left(\frac{\frac{x_\ell-k^{(n)}_\ell}{2k^{(n)}_\ell-x_\ell}}{1+\frac{x_\ell-k^{(n)}_\ell}{2k^{(n)}_\ell-x_\ell}}\right)\right|
\;=\;
\left|\frac{\partial }{\partial x_\ell}\left(\frac{x_\ell}{k^{(n)}_\ell}-1\right)\right|
\;=\;
\frac{1}{k^{(n)}_\ell}\;.
\end{multline}
From \eqref{eq:u_prime} and \eqref{eq:ineq_u_prime} it follows
that$\left|\frac{\partial u^{(n)}}{\partial x_\ell}(x)\right|\leq\frac{1}{k^{(n)}_\ell}$ for any $x$ in $\Reals^\infty$,
and thus
$\Expect{\left|\frac{\partial u^{(n)}}{\partial x_\ell}(\X)\right|^2}\leq\left(\frac{1}{k^{(n)}_\ell}\right)^2$.
Therefore 
$$
\sum_{\ell=1}^\infty\Expect{\left|\frac{\partial u^{(n)}}{\partial x_\ell}(\X)\right|^2}
\quad\leq\quad
\sum_{\ell=1}^\infty\left(\frac{1}{k^{(n)}_\ell}\right)^2
\;\leq\;
\sum_{\ell=1}^\infty\frac{1}{n^2\ell^2}
\;=\;
\sum_{\ell=1}^\infty\frac{1}{n^2\ell^2}
\;=\;
\frac{\pi^2}{6n^2}
\;\rightarrow\;
 0\,,
$$
and thus $\Phi(u^{(n)})\rightarrow 0$, which completes the proof.
\end{proof}


\bibliographystyle{chicago}
\bibliography{Optimal} 


\end{document}